\documentclass[11pt]{article}
\def\thetitle{The diameter of the uniform spanning tree of dense graphs}

\usepackage{amsmath,amssymb}

\usepackage[usenames,dvipsnames,svgnames,table]{xcolor}
\definecolor{CombinatoricaAqua}{HTML}{00698C}
\definecolor{CombinatoricaBlue}{HTML}{3A3293}
\definecolor{CombinatoricaBrown}{HTML}{66220C}
\definecolor{CombinatoricaRed}{HTML}{DF2A27}
\definecolor{HarvardCrimson}{rgb}{0.6471, 0.1098, 0.1882}

\makeatletter
\let\reftagform@=\tagform@
\def\tagform@#1{\maketag@@@
	{(\ignorespaces\textcolor{CombinatoricaBrown}{#1}\unskip\@@italiccorr)}}
\renewcommand{\eqref}[1]{\textup{\reftagform@{\ref{#1}}}}
\makeatother

\usepackage[backref=page]{hyperref}
\hypersetup{
	unicode,
	pdfencoding=auto,
	pdfauthor={Noga Alon, Asaf Nachmias and Matan Shalev},
	pdftitle={\thetitle},
	pdfsubject={},
	pdfkeywords={},
	colorlinks=true,
	citecolor=CombinatoricaBlue,
	linkcolor=CombinatoricaAqua,
	anchorcolor=CombinatoricaBrown,
	urlcolor=HarvardCrimson}

\usepackage{amsthm}
\usepackage{thmtools}
\usepackage{thm-restate}
\usepackage{bbm}
\usepackage{nicefrac}
\usepackage[capitalize]{cleveref}
\Crefname{fact}{Fact}{Facts}
\Crefname{claim}{Claim}{Claims}
\Crefname{assumption}{Assumption}{Assumptions}


\declaretheoremstyle[
spaceabove=\topsep, spacebelow=\topsep,
headfont=\color{CombinatoricaBrown}\normalfont\bfseries,
bodyfont=\itshape,
]{thm}
\declaretheoremstyle[
spaceabove=\topsep, spacebelow=\topsep,
headfont=\color{CombinatoricaBrown}\normalfont\bfseries,
bodyfont=\normalfont,
]{dfn}
\declaretheoremstyle[
spaceabove=0.5\topsep, spacebelow=0.5\topsep,
headfont=\color{CombinatoricaBrown}\normalfont\bfseries,
bodyfont=\normalfont,
]{rmk}

\declaretheorem[style=thm,parent=section]{theorem}
\declaretheorem[style=thm,sibling=theorem]{lemma}

\declaretheorem[style=thm,sibling=theorem]{claim}

\declaretheorem[style=rmk,sibling=theorem]{remark}

\declaretheorem[style=definition,sibling=theorem]{definition}

\usepackage[nobysame,msc-links]{amsrefs}

\BibSpec{book}{%
	+{}  {\PrintPrimary}                {transition}
	+{,} { \textbf}                     {title} 
	+{.} { }                            {part}
	+{:} { \textit}                     {subtitle}
	+{,} { \PrintEdition}               {edition}
	+{}  { \PrintEditorsB}              {editor}
	+{,} { \PrintTranslatorsC}          {translator}
	+{,} { \PrintContributions}         {contribution}
	+{,} { }                            {series}
	+{,} { \voltext}                    {volume}
	+{,} { }                            {publisher}
	+{,} { }                            {organization}
	+{,} { }                            {address}
	+{,} { \PrintDateB}                 {date}
	+{,} { }                            {status}
	+{}  { \parenthesize}               {language}
	+{}  { \PrintTranslation}           {translation}
	+{;} { \PrintReprint}               {reprint}
	+{.} { }                            {note}
	+{.} {}                             {transition}
	+{}  {\SentenceSpace \PrintReviews} {review}
}

\makeatletter
\renewcommand{\PrintNames@a}[4]{%
	\PrintSeries{\name}
	{#1}
	{}{ and \set@othername}
	{,}{ \set@othername}
	{}{ and \set@othername}
	{#2}{#4}{#3}%
}
\makeatother

\makeatletter
\def\mathcolor#1#{\@mathcolor{#1}}
\def\@mathcolor#1#2#3{%
	\protect\leavevmode
	\begingroup
	\color#1{#2}#3%
	\endgroup
}
\makeatother
\definecolor{Red}{rgb}{0.618,0,0}
\definecolor{Blue}{rgb}{0,0,1}
\definecolor{Green}{rgb}{0,0.298,0}

\usepackage{sectsty}
\sectionfont{\color{CombinatoricaBrown}}
\subsectionfont{\color{CombinatoricaBrown}}
\subsubsectionfont{\color{CombinatoricaBrown}}

\usepackage{soul}
\soulregister\ref7

\usepackage[textsize=tiny,backgroundcolor=black!10]{todonotes}
\presetkeys{todonotes}{noline}{}

\usepackage{pifont}

\usepackage{comment}

\usepackage[a4paper]{geometry}
\title{\thetitle}
\author{Noga Alon \and Asaf Nachmias \and Matan Shalev}

\makeatletter
\def\namedlabel#1#2{\begingroup
  #2%
  \def\@currentlabel{#2}%
  \phantomsection\label{#1}\endgroup
}
\makeatother

\usepackage{mleftright}
\mleftright

\newcommand{\defn}[1]{{\bfseries #1}}

\newcommand{\eps}{\varepsilon}
\newcommand{\NN}{\mathbb{N}}

\newcommand{\ZZ}{\mathbb{Z}}
\newcommand{\cB}{\mathcal{B}}

\newcommand{\cT}{\mathcal{T}}

\newcommand{\LERW}{\mathsf{LERW}}

\newcommand{\floor}[1]{\left\lfloor{#1}\right\rfloor}
\newcommand{\ceil}[1]{\left\lceil{#1}\right\rceil}

\newcommand{\es}{\varnothing}

\DeclareMathOperator{\diam}{diam}

\newcommand{\vect}{\mathbf}

\newcommand{\pr}[0]{\mathbb{P}}
\newcommand{\E}[0]{\mathbb{E}}

\DeclareMathOperator{\proband}{and}

\newcommand{\pand}[0]{\ \proband \ }

\newcommand{\dtv}{d_{\mathrm{TV}}}
\newcommand{\tv}{\mathrm{TV}}
\newcommand{\ind}{\mathbf{1}}

\DeclareMathOperator{\vol}{Vol}
\newcommand{\UST}{\mathsf{UST}}

\newcommand{\tmix}{t_{\mathrm{mix}}}

\newcommand{\LE}{\mathsf{LE}}

\newcommand{\bub}{\cB}

\newcommand{\deco}{\mathcal{P}}
\DeclareMathOperator{\Vol}{Vol}

\AtEndDocument{\bigskip{\footnotesize%
		\par
	
		\textsc{Noga Alon} \par
		\textsc{Department of Mathematics, Princeton University, Princeton, NJ 08544, USA and Schools of Mathematics and Computer Science, Tel Aviv University, Tel Aviv 69978, Israel} \par
		\textit{Email:} 
		\href{mailto:asafnach@tauex.tau.ac.il}{\texttt{nogaa@tau.ac.il}} \par
		\addvspace{\medskipamount}

		\par
		
		\textsc{Asaf Nachmias} \par
		\textsc{School of Mathematical Sciences, Tel Aviv University, Tel Aviv 69978, Israel} \par
		\textit{Email:} 
		\href{mailto:asafnach@tauex.tau.ac.il}{\texttt{asafnach@tauex.tau.ac.il}} \par
		\addvspace{\medskipamount}
		
		\par
		
		\textsc{Matan Shalev} \par
		\textsc{School of Mathematical Sciences, Tel Aviv University, Tel Aviv 69978, Israel} \par
		\textit{Email:}
		\href{mailto:matanshalev@mail.tau.ac.il}{\texttt{matanshalev@mail.tau.ac.il}} \par
		\addvspace{\medskipamount}
				
}}

\begin{document}
\maketitle
\begin{abstract} We show that the diameter of a uniformly drawn spanning tree of a simple connected graph on $n$ vertices with minimal degree linear in $n$ is typically of order $\sqrt{n}$. A byproduct of our proof, which is of independent interest, is that on such graphs the Cheeger constant and the spectral gap are comparable. 
\end{abstract}
\section{Introduction}\label{sec:introduction}
The \defn{uniform spanning tree} of a finite connected graph $G$, denoted by $\UST(G)$, is a uniformly chosen random spanning tree of $G$.  
The main result of this paper is that the \defn{diameter} of the $\UST$, i.e., the largest distance between two vertices of the $\UST$, on graphs with linear minimal degree grows like the square root of the number of vertices with high probability.

When $G$ is the complete graph on $n$ vertices, $K_n$, much more is known. A classical result of Szekeres \cite{Sze83} (see also \cite{Kolchin}) explicitly provides the limiting distribution of the diameter of $\UST(K_n)$ scaled by $n^{-1/2}$. This was greatly extended by the influential work of Aldous \cites{ACRT1,ACRT2,ACRT3} and Le Gall \cite{LeGall05,LeGall06} who proved that $\UST(K_n)$, viewed as a random metric space and scaled by $n^{-1/2}$, converges in distribution with respect to the Gromov-Hausdorff distance to a canonical random compact metric space known as the Continuum Random Tree \cite{ACRT1}. 

The $\UST$ is a critical statistical physics model, hence it is expected that as long as the base graph $G$ is ``high dimensional'', $\UST(G)$ should have a similar geometry to that of $\UST(K_n)$. This high dimensionality condition is typically some good isoperimetric condition. This has been pursued in \cite{MNS19+} where the authors show that the diameter of $\UST(G)$ is of order $\sqrt{n}$ for a large class of high dimensional graphs including, for example, $\ZZ^5_n$, the hypercube $\{0,1\}^n$ and regular expanders. The dense graphs we study in this paper, however, can be very far from being high dimensional. For instance, two cliques on $n/2$ vertices connected by an edge will have the worst isoperimetric inequality, yet the diameter of its $\UST$ is still of order $\sqrt{n}$. We now state our main result. For a connected graph $H$ we write $\diam(H)$ for the maximal graph distance in $H$ between any two vertices. 
\begin{theorem}\label{thm:main}
	For any $\eps, \delta\in(0,1)$ there exists $C=C(\delta,\eps)\in (1,\infty)$ such that if $G$ is a connected simple graph on $n$ vertices with minimal degree at least $\delta n$, then,
	\begin{equation*}
		\pr\left(C^{-1}\sqrt{n}\le\diam(\UST(G))\le C\sqrt{n}\right) \ge 1-\eps \, ,
	\end{equation*}
\end{theorem}

The main tool we use is a decomposition theorem (\cref{lem:main:deco}) which can be thought of as Szemer\'edi-type Regularity Lemma allowing to partition the vertices of $G$ into $O(1)$ sets such that the induced graph on each set satisfies a sufficiently strong isoperimetric inequality and such that the number of edges connecting two such sets is sufficiently small. This partition is then used to study the behavior of the loop-erased random walk on $G$ which in turn provides estimates on the $\UST$ via Wilson's algorithm (\cref{subsec:prelim}).

It turns out that one can get significant mileage in the study of the random walk using such a decomposition theorem. One such estimate, which we believe is of independent interest, is an improvement to Cheeger's inequality on graphs of linear minimal degree. This improved inequality shows that on such graphs the Cheeger constant and the spectral gap are comparable. Denote by $P$ the transition matrix of the simple random walk on $G$, and let $\pi(v)=\deg(v)/2|E(G)|$ denote its stationary distribution. Since $P$ is self-adjoint in $L^2(\pi)$ it has $n$ real eigenvalues in $[-1,1]$ denoted by \begin{equation*}
	1 = \lambda_1 \geq \lambda_2 \geq ... \geq \lambda_n \geq -1.
\end{equation*}
A classical highly useful inequality proved by Alon-Milman \cites{AM85,A86}, Lawler-Sokal \cite{LS88} and Jerrum-Sinclair \cite{JS89} known as \defn{Cheeger's inequality} relates the \defn{spectral gap} $\gamma(G):=1-\lambda_2$ of $P$ with its isoperimetric constant (also known as Cheeger's constant). More precisely, for a set of vertices $S$ of $G$ we denote its volume by $\Vol(S) = \sum_{v\in S}\deg(s)$ and its edge boundary by $\partial S = \{(u,v)\in E(G) \mid u\in S, v\notin S\}$. We define the Cheeger constant as
\begin{equation*}
	\Phi(G) := \min_{S, \pi(S) \leq 1/2}\frac{|\partial S|}{\Vol(S)}
\end{equation*}
Cheeger's inequality states that
\begin{equation}\label{eq:cheeger}
	\Phi(G)^2/2 \leq \gamma(G) \leq 2\Phi(G).
\end{equation}  
When $G$ is a simple graph of linear minimal degree we can improve the lower bound in Cheeger's inequality to match the order of the upper bound.

\begin{restatable}[]{theorem}{cheegermix}\label{thm:cheeger:is:everything}
	For any $\delta\in(0,1)$ there exists a constant $c(\delta)>0$ such that the following holds. Let $G=(V,E)$ be a simple graph on $n$ vertices with minimal degree at least $\delta n$ and Cheeger constant $\Phi(G)$, then 
	$$ \gamma(G) \geq c(\delta) \Phi(G) \, .$$
\end{restatable}

\begin{remark} Our proof gives $c(\delta) = \delta^{19}/2^{34}$ but we have not tried to optimize this constant.
\end{remark}

\begin{remark}\label{rmrk:alternate}
	After posting this paper we learned from Majid Farhadi, Suprovat Ghoshal, Anand Louis, and Prasad Tetali of an alternate proof of \cref{thm:cheeger:is:everything}, which gives $c(\delta) = \Omega(\delta)$. Since the proofs are completely different we believe there is value in presenting both. We emphasize that our proof of \cref{thm:cheeger:is:everything} is just a byproduct of the tools we develop to prove our main result \cref{thm:main} and is also quite short given these tools. It is presented in \cref{sec:cheegereverything}.

	The alternate proof follows from Theorem 2 of \cite{KLLGT13}. In our notation, it states that there exists a universal constant $C>0$ such that \begin{equation}\label{eq:kklgt}
	\Phi(G) \leq \frac{Ck\gamma(G)}{\sqrt{1-\lambda_k}}.
	\end{equation}
	To obtain \cref{thm:cheeger:is:everything} from \eqref{eq:kklgt}, observe that when the minimal degree of $G$ is at least $\delta n$, each diagonal term of $P^2$ is bounded from above by $1/(\delta n)$ and therefore the trace of $P^2$ is bounded from above by $1/\delta$. Since the trace of $P^2$ equals the sum of squares of the eigenvalues of $P$, we have that at most $4/\delta$ eigenvalues of $P$ are larger than $1/2$. Hence, $\lambda_{\ceil{4/\delta}} \leq 1/2$. Plugging this into \eqref{eq:kklgt} and rearranging gives
	\begin{equation*}
	\gamma(G) \geq c'\delta\Phi(G)
	\end{equation*}
	for some  $c'>0$.
\end{remark}

\subsection{Preliminaries}\label{subsec:prelim}

For a finite graph $G = (V,E)$ and a vertex $v\in V$ we denote by $\deg_G(v)$ its degree. When $U\subseteq V$ we will write $\deg_G(v,U)$ for the number of edges between $v$ and $U$. We will sometimes omit the subscript when it is obvious to which $G$ we refer. A \defn{network} $(G,w)$ is a connected graph $G=(V,E)$ endowed with a non-negative function $w:E\to [0,\infty)$ on its edges. The \defn{simple random walk} $\{X_t\}_{t=0}^{\infty}$ on $G$ is the Markov chain on the state space $V$ that at each step moves along a uniformly chosen edge incident to it. Similarly, the simple random walk on a network $(G,w)$ is the Markov chain such that the transition probability from $v$ to $u$ is proportional to $w\left(\{v,u\}\right)$. In order to avoid issues of parity, we will sometimes consider the \defn{lazy random walk}. Formally, at each step, with probability $1/2$ the walker stays put and otherwise chooses a neighbour uniformly (or proportionally to $w(\{v,\cdot\})$. We will often consider random walks with different starting distributions. When $\mu$ is a probability measure on $V$, we will use the notation $\pr_\mu$ for the probability measure conditioned on $X_0 \sim \mu$. We will also use $\pr_v$ for the walk conditioned on $X_0 = v$. Also, for a non negative integer $t\geq 0$ and two vertices $v,u\in V$ we write $\vect{p}^t(v,u)$ for $\pr_v(X_t = u)$. 
For any $a,b\in[0,\infty)$, we write $X[a,b]$ for $\langle X_i\rangle_{i\in I}$ where $I=\left[\ceil{a},\floor{b}\right]\cap\NN$. Similarly, we write $X[a,b)$ if we wish to exclude $b$ from $I$.

We will frequently use some facts about the {\em mixing time} of the random walk on $G$ which we now define. The \defn{total variation distance} between two probability measures $\mu,\nu$ on $V$ is 
\begin{equation*}
	\dtv(\mu,\nu) := \frac{1}{2}\sum_{u\in V}|\mu(v)-\nu(v)|.
\end{equation*} 
For every $\eps \in (0,1/2)$, the \defn{$\eps$-mixing time} of $G$ is defined by
\begin{equation*}
	\tmix^G(\eps) := \max_{v\in V}\min\left\{t\geq 0 : \|\vect{p}^t(v,\cdot) - \pi(\cdot) \|_\tv < \eps \right\},
\end{equation*}
where $\pi(v) = \deg(V)/2|E|$, the stationary distribution of the random walk on $G$. To avoid issues of periodicity we emphasize that in this paper the quantity $\tmix^G(\eps)$ is defined \emph{only} for the lazy random walk. We liberally vary the choice of $\eps$ throughout the proof; this changes the mixing time by at most a multiplicative constant. Indeed, for every $\eps<1/4$ and any integer, we have (see, \cite{LPW2}*{Eq 4.34 and Eq 4.32})
\begin{equation}\label{eq:exp:mix}
	\tmix^G(\eps) \leq \log_2(\eps^{-1})\tmix(1/4) \qquad \hbox{and} \qquad  \tmix^G(\eps^k) \leq k\tmix(\eps/2).
\end{equation}
\\

The uniform spanning tree ($\UST$) of $G$ is the uniform measure over the set of all spanning trees of $G$. More generally, when $(G,w)$ is a finite network, we denote by $\UST(G)$ the weighted uniform spanning tree. That is, the probability measure supported on spanning trees of $G$ that assigns to each such tree $T$ a measure proportional to $\prod_{e\in T}w(e)$. We briefly describe here some useful properties of the $\UST$ involving sampling, conditioning and stochastic domination and refer the reader to \cite[Chapter 4]{LP} for a comprehensive overview.

Our analysis of the $\UST$ will rely on \defn{Wilson's algorithm} \cite{Wil96} for efficiently sampling the $\UST$. This popular algorithm is frequently used  not just to sample but rather to prove theorems about the $\UST$, see \cite{LP}.  
Let $G=(V,E)$ be a finite connected graph. A walk of length $L$ on $G$ is a sequence of vertices $(X_0,\ldots,X_L)$ such that $(X_{i},X_{i+1}) \in E$ for every $0 \leq i < L$. 
Given such a walk $X$, its \defn{loop-erasure} $\LE(X)$ is a sequence of vertices defined as follows. We put $\LE(X)_0 = X_0$ and inductively, for every $i>0$ and given $\LE(X)[0,i-1]$, define $s_i := \max\{t \leq L \mid X_t = \LE(X)_{i-1}\}$. If $s_i = L$, the loop erased random walk of $X$ is $(\LE(X)_0,\ldots,\LE(X)_{i-1})$. Otherwise, let $\LE(X)_i = X_{s_i + 1}$. In words, we walk along $(X_0,\ldots,X_L)$ and erase the loops as they are formed. Given two vertices $v,u$ the \defn{loop erased random walk} from $v$ to $u$ is defined to be $\LE(X)$ where $X$ is the simple random walk started at $v$ and terminated upon when hitting $u$. In a similar fashion we define the loop erased random walk from a vertex $v$ to a subset of vertices $U$. Wilson's algorithm works as follows. Choose any ordering $(v_1,\ldots,v_n)$ of the vertices of $G$ and let $T_1$ be the empty tree containing $v_1$ and no edges. At each step $i>1$, run a loop erased random walk from $v_i$ to $T_{i-1}$ let $T_i$ be the union of this loop erased random walk and $T_{i-1}$. This process terminates after going through all vertices and results a spanning tree $T_n$. A remarkable theorem of Wilson \cite{Wil96}, that we use throughout this paper, states that $T_n$ is distributed as $\UST(G)$.\\

Next, it is very simple to prove that conditioning on the existence or absence of edges in the $\UST$ results in a $\UST$ on the graph obtained from $G$ by contracting or erasing those edges, respectively.
\begin{lemma}\label{lem:contract} \cite[Section 4.2]{LP}
	Let $(G,w)$ be a network and let $e$ be an edge of $G$. The $\UST$ of $(G,w)$ conditioned to contain $e$ is distributed as the union of $e$ with the $\UST$ of the network obtained from $G$ by contracting the edge $e$ to a single vertex.
\end{lemma}

Lastly, we recall a few highly useful corollaries  to a result of Feder and Mihail \cite{FM92}. Given two probability measures $\mu_1$ and $\mu_2$ on $2^E$, we say that $\mu_1$ is \defn{stochastically dominated} by $\mu_2$ if there exists a probability measure $\mu$ on $2^E\times 2^E$ with marginals $\mu_1$ and $\mu_2$ which is supported on
\begin{equation*}
	\{(T_1, T_2) \in 2^E\times 2^E \mid T_1\subseteq T_2\}.
\end{equation*} 

\begin{lemma}\label{lem:neg:ass} \cite[Lemma 10.3]{LP}
	Let $G$ be a connected subgraph of a finite connected graph $H$. 
	Then, $\UST(H)\cap E(G)$ is stochastically dominated by $\UST(G)$ when both are viewed as probability measures on $2^{E(G)}$.
\end{lemma}

This lemma can be further generalized to our needs. We say that a network $(G,w)$ is a subnetwork of $(H,w')$ if $V(G) \subseteq V(H)$ and for every edge $(v,u)$ with $w(v,u) \neq 0$ we have $w(v,u) = w'(v,u)$. The same proof of \cite[Lemma 10.3]{LP} yields a more general statement.

\begin{lemma}\label{lem:neg:ass:net}
	Let $(G,w)$ be a subnetwork of a finite network $(H,w')$. Then, $\UST(H)\cap E(G)$ is stochastically dominated by $\UST(G)$ when both are viewed as probability measures on $2^{E(G)}$.
\end{lemma}

 Given a network $G$ and a subset of vertices $A$ we write $G/A$ for the network obtained from $G$ by contracting the vertices of $A$ to a single vertex and keeping all edges. The following is well known and can be obtained by a similar argument to the proof of \cite[Lemma 10.3]{LP}.

\begin{lemma}\label{lem:neg:ass:contract}
	Let $(G,w)$ be a finite network and let $A\subseteq B$ be two sets of vertices of $G$. Then, $\UST(G/A)$ stochastically dominates $\UST(G/B)$.
\end{lemma}

\subsection{Proof outline and organization}

It is easier to bound the diameter of the $\UST$ after conditioning on a long path in it. Indeed, a key lemma from \cite{MNS19+} (see \cref{lem:mns:lemma}) roughly states that if a vertex set $W\subset V$ is sufficiently spread out in the sense that the random walk is unlikely to avoid it, then one can upper bound the probability that the diameter of $\UST(G/W)$ is much larger than $|W|$, where $G/W$ is the graph obtained from $G$ by identifying $W$ to a single vertex. When $G$, say, is a regular expander (or any other ``high dimensional'' graph) the approach in \cite{MNS19+} is to take $W$ to be the vertices on the unique path in $\UST(G)$ between two independently drawn uniform vertices of $G$. The expansion property is then used to show that this set has size $\Theta(\sqrt{n})$ and is sufficiently spread out, so \cref{lem:mns:lemma} implies that $\UST(G/W)$ has diameter $\Theta(\sqrt{n})$.

The high level approach in this paper is to use our decomposition theorem (\cref{lem:main:deco}, proved in \cref{sec:dec}) and partition the graph into $O(1)$ sets so that with high probability the $\UST$ path between two random vertices in each set remains within the set, is of size $\Theta(\sqrt{n})$ and is sufficiently spread out within the set. Formalizing and proving this is performed in \cref{sec:rwdeco}. In \cref{sec:proofmainthm} we take the union of these $O(1)$ paths to be our set $W$ and apply \cref{lem:mns:lemma} from \cite{MNS19+} to obtain that $\UST(G/W)$ has diameter roughly of order $\sqrt{n}$ from which we deduce the desired upper bound on the diameter of $\UST(G)$.

\section{Decomposition of linear minimal degree graphs}\label{sec:dec}
In this section we prove that any finite connected graph $G=(V,E)$ on $n$ vertices with linear minimal degree can be decomposed into $k$ sets, each of them linear in the number of vertices, such that a random walk typically mixes inside every such set before leaving it. Hence, when considering short times, roughly $\sqrt{n}$ steps of the walk, a random walk on the graph can be approximated well by a random walk on one of its sets in this decomposition. We will denote a partition of $V$ by $\deco$ and sometimes more explicitly by $V=V_1\sqcup\ldots\sqcup V_k$. We write $[k]$ for the set $\{1,\ldots,k\}$.	

\begin{definition}\label{def:good:deco} Let $\eps\in(0,1),\delta\in(0,1]$ and $\beta>0$ be fixed and let $G=(V,E)$ be a graph on $n$ vertices with minimal degree at least $\delta n$. We say that that a partition $V=V_1 \sqcup \ldots \sqcup V_k$ is an \defn{$(\eps,\delta,\beta)$-good decomposition}
	if there exists some $\theta \in [\eps^{11\cdot 2^{2/\delta}},\eps]$ such that the following conditions are satisfied.
	\begin{enumerate}
		\item The number of sets in the decomposition, denoted by $k$, satisfies $k \leq 2/\delta$.
		\item For every $i \in[k]$ we have $|V_i| \geq \frac{\delta n}{2}$.
		\item For every $i \in [k]$, the spectral gap of $G[V_i]$ is at least $\frac{\delta^{15}\theta \beta}{2^{31}n^2}$.
		\item For every $i\in[k]$ and each $v\in V_i$ we have that $\deg(v,V_i) \geq \frac{\delta^4 n}{40}$. 
		\item For every $i\in[k]$  we have $|E(V_i,V\setminus V_i)| \leq \eps^9\theta^2\beta.$
	\end{enumerate}
\end{definition}

\noindent The majority of this section is devoted to proving the following decomposition lemma which will be key in the proof of \cref{thm:main}.

\begin{lemma}\label{lem:main:deco}
	For every  $\delta>0$, there exists a constant $c=c(\delta)>0$ such that the following holds. For every $\eps\in (0,c)$, every $\beta\in(0,240n^2/(\eps\delta^4))$ and any simple graph $G=(V,E)$ on $n$ vertices with minimal degree at least $\delta n$ there exists an $(\eps,\delta,\beta)$-good decomposition of $G$.
\end{lemma}

We remark that the proof of \cref{thm:cheeger:is:everything} does not use this lemma, rather a simpler decomposition lemma, \cref{cl:prim:deco}, which is also the first step in the proof of \cref{lem:main:deco}. 

\subsection{Preliminary estimates on the spectral gap} \label{sssec:spec:gap}

In this subsection we prove the following lemma allowing us to lower bound the spectral gap of a decomposable graph. 
\begin{definition}\label{def:hp} Given a partition $\deco$ of $V$, denoted by $V=V_1 \sqcup \ldots \sqcup V_k$, and some $c>0$, we define the graph $H(\deco,c)$ as follows. The vertices of $H(\deco,c)$ are $[k]$ where each vertex $i\in[k]$ represents a set $V_i$ of $\deco$ and we join an edge $(i,j)$ if $|E(V_i,V_j)| > c$.
\end{definition}

\begin{lemma}\label{lem:dec:spec:gap}
	Let $G=(V,E)$ be a simple graph on $n$ vertices. Let $\deco$ be a partition of $V$ denoted by $V=V_1\sqcup\ldots\sqcup V_k$. Assume that there exists $a,b,c>0$ such that the following conditions hold.
	\begin{itemize}
		\item For every $i\in[k]$, the spectral gap of $G[V_i]$ is larger than $a$,
		\item For every $i\in[k]$ and every $v\in V_i$ we have $\deg(v,V_i) \geq b$,
		\item The graph $H(\deco,c)$ is connected.
	\end{itemize}
	Then, the spectral gap of $G$ is at least $\min\left\{a, \frac{abc}{6kn^3}\right\}$. 
\end{lemma}

\cref{lem:dec:spec:gap} is an application of the main result of \cite{JSVT04} together with some quick estimates involving the Dirichlet form (see \cite[Chapter 13]{LPW2} for further reading on the Dirichlet form). In the rest of this subsection we cite and prove these necessary background results, then prove the lemma. Since the proof digresses from the main ideas of this paper, the reader may want to take this lemma as a ``black box'' and skip reading its proof. We will typically use this Lemma when $a$ is roughly a constant, $b$ is of order $n$ and $c$ is of order $\Phi(G)n^2$.

Let $(G,w)$ be a network where $G = (V,E)$ with a partition of its vertex set $V = V_1 \sqcup \ldots \sqcup V_k$. Let $\pi$ be the the stationary measure of the simple (or lazy) random walk on $(G,w)$.
For such a network with a partition of it vertex set to $k$ sets,
we define the distribution $\overline{\pi}$ on $[k]$ by setting $\overline{\pi}(i) = \sum_{v\in V_i}\pi(v)$. The \defn{projection chain} is a Markov chain on $[k]$ with the following transition probabilities
\begin{equation}\label{eq:proj:chain}
\overline{P}_{i,j} = \frac{1}{\overline{\pi}(i)}\sum_{v\in V_i, u\in V_j}\pi(v)P(v,u).
\end{equation}  
Note that $\overline{\pi}$ is the stationary distribution of this chain.
Furthermore, we define $k$ \defn{restriction chains}, to which we will also refer as the restriction walks. For every $i\in[k]$, this restriction walk is a Markov chain on $V_i$ with transition probabilities
\begin{equation*}
P_i(x,y) = \begin{cases}
P(x,y) & x\neq y, \\
1 - \sum_{w\in V_i\setminus \{x\}}P(x,w) & x=y.
\end{cases}
\end{equation*} 
We are now ready to state a weaker version of \cite{JSVT04}*{Theorem 1} which will be useful later. We remark that our version follows easily from \cite{JSVT04}*{Theorem 1} by applying trivial upper bounds to the spectral gap of the projection chain and to the probability to move from one set in the decomposition to another one.
\begin{theorem}[\cite{JSVT04}*{Theorem 1}]\label{thm:jsvt}
	Let $(G,w)$ be a network where $G=(V,E)$ and let $V=V_1 \sqcup \ldots \sqcup V_k$ be a partition of its vertex set. Denote by $\bar{\gamma}$ the spectral gap of the projection chain associated with it and for every $i\in[k]$ let $\gamma_i$ be the spectral gap of the restriction walk on $V_i$. Then, the spectral gap $\gamma$ of the random walk on $(G,w)$ satisfies
	\begin{equation*}
	\gamma \geq \min_i \frac{\bar{\gamma}\gamma_i}{6}.
	\end{equation*} 
\end{theorem}

For an irreducible Markov chain on a finite state space $\Omega$ with transition matrix $P$, stationary distribution $\pi$ and a function $f:\Omega \to \mathbb{R}$, we denote
\begin{equation*}
\mathcal{E}^P_\pi(f) := \frac{1}{2}\sum_{x,y\in \Omega}\pi(x)P(x,y)(f(x)-f(y))^2.
\end{equation*}

The following lemma which we will not prove is helpful in estimating spectral gaps of networks which are obtained by a small perturbation of another network.

\begin{lemma}[\cite{LPW2}*{Lemma 13.8}]\label{lem:comp:energy}
	Let $P_0$ and $P_1$ be transition matrices with stationary distributions $\pi_0$ and $\pi_1$ over the same finite state space $\Omega$. Let $\gamma^0$ and $\gamma^1$ be their spectral gaps, respectively. If there exists $\alpha > 0$ such that for all functions $f:\Omega\to\mathbb{R}$ we have $\mathcal{E}^{P_1}_{\pi^1}(f) \leq \alpha\mathcal{E}^{P_0}_{\pi^0}(f)$, then 
	\begin{equation*}
	\gamma^1 \leq \left(\max_{\omega\in \Omega}\frac{\pi^0(\omega)}{\pi^1(\omega)}\right) \alpha \gamma^0.
	\end{equation*}
\end{lemma}

\begin{claim}\label{cl:dec:gap}
	Let $G=(V,E)$ and let $W_0 = (G,w_0)$ and $W_1 = (G,w_1)$ be two networks such that there exists a vertex $v\in V$ such that $w^0_{v,v} < w^1_{v,v}$ and for every other edge $(u,w) \neq (v,v)$ we have $w^0_{u,w} = w^1_{u,w}$. Denote by $\gamma^0$ and $\gamma^1$ the spectral gaps corresponding to $W_0$ and $W_1$, respectively. Then, $\gamma^1 \leq \gamma^0$. 
\end{claim}
\begin{proof}
	For $i\in \{0,1\}$, we let $P_i$ be the transition matrix corresponding to $W_i$. We denote 
	\begin{align*}
	w^i_u = \sum_{e\ni u}w^i_e, \quad \quad Z_i = \sum_{u\in V} w^i_u.
	\end{align*}
	We also denote by $\pi^i$ the stationary distribution of $W_i$ and recall that $\pi^i(u) = w^i_u / Z_i$. 
	We will now use \cref{lem:comp:energy} to show that $\gamma^1 \leq \gamma^0$. A simple calculation shows that for every $f$
	\begin{equation}\label{eq:energy:ineq}
	\mathcal{E}^{P_1}_{\pi^1}(f) = \frac{1}{2}\sum_{u,v\in V}\frac{w^1_{u,v}}{Z_1}(f(u)-f(v))^2 =  \frac{Z_0}{Z_1}\mathcal{E}^{P_0}_{\pi^0}(f).
	\end{equation}
	Also, we write $\eps = w^1_{vv} - w^0_{vv} > 0$. Then, for every $u \in V$ we have
	\begin{equation*}
	\frac{\pi^0(u)}{\pi^1(u)} = \frac{\frac{w_u^0}{Z_0}}{\frac{w_u^1}{Z_1}} = \begin{cases}
	\frac{Z_1}{Z_0} & u\neq v \\
	\frac{Z_1}{Z_0}\cdot \frac{w_v^0}{w_v^0 + \eps} & u = v
	\end{cases}.
	\end{equation*}
	Hence, in any case $\pi^0(u)/\pi^1(u) \leq Z_1/Z_0$. By \eqref{eq:energy:ineq}, we can use \cref{lem:comp:energy} with $\alpha = Z_0 / Z_1$. We thus get that $\gamma^1 \leq \gamma^0$. 
\end{proof}

As mentioned in the last subsection, if $P$ is a transition matrix of some random walk and $Q$ is the transition matrix of its lazy version, then $Q = \frac{1}{2}(I+P)$ and hence $\gamma(Q) = \frac{1}{2}\gamma(P)$. Similarly, if $X$ is a random walk with transition $P$ and $Y$ is an $\alpha$-lazy random version of $X$, that is, at each step the walker stays put with probability $\alpha$ and otherwise chooses its next step according to $P$, then $Q = \alpha I + (1-\alpha)P$. Hence, $\gamma(Q) = (1-\alpha)\gamma(P)$. 

We can further generalize this. For every $v\in V$, let $p_v\in[0,1)$. We call $(p_v)_{v\in V}$ the \defn{lazy vector}. Let $X$ be some random walk on $V$ with transition matrix $P$ and let $Y$ be the following random walk on $V$. At each step, if the walker is at some $u\in V$, it stays put with probability $p_u$ and  otherwise chooses its next step according to $P$. The next claim shows that we can lower bound the spectral gap associated with this random walk.

\begin{claim}\label{cl:weird:laziness}
	Let $P$ be a transition matrix of a random walk $X$ on a finite connected graph $G$ with spectral gap $\gamma(P)$. Let $\alpha \in (0,1)$ and let $(p_v)_{v\in V} \in [0,\alpha)^V$ be a vector which we call the lazy vector. Let $Y$ be the following random walk on $G$. If $Y_{t} = v$, stay put with probability $p_v$. Otherwise, choose $X_{t+1}$ according to $P$. Let $Q$ be the transition matrix of $Y$. Then, $\gamma(Q) \geq (1-\alpha)\gamma(P)$.
\end{claim}

\begin{proof}
	Let $(G,w^\alpha)$ be the network associated with the graph $G$ such that $w^\alpha_{uv} = 1$ if $(u,v)\in E$ and $w^\alpha_{vv} = \beta_v$ where $\beta_v$ satisfies $\frac{\beta_v}{\deg(v) + \beta_v} = \alpha$. If $P$ is the transition matrix of the simple random walk on $G$, then $Q:= \alpha I + (1-\alpha)P$ is the transition matrix corresponding to $(G,w^\alpha)$. Note that the spectral gap of $Q$ satisfies $\gamma(Q) = (1-\alpha)\gamma(P)$. Let $(p_v)_{v\in V}$ be the lazy vector with all values non-negative and smaller than $\alpha$. The random walk that stays put at some $u\in V$ with probability $p_u$ and jumps according to $P$ otherwise can be seen as a random walk on a network $(G,w^p)$ which can be constructed from $(G,w^\alpha)$ by going iteratively over all vertices $v\in V$ and decreasing the weight of $w^\alpha_{vv}$ at each step according to $p_v$.  By \cref{cl:dec:gap}, at each step the spectral gap can be only increased. Hence, the spectral gap of the walk corresponding to $(G,w^p)$ is larger than $\gamma(Q) = (1-\alpha)\gamma(P)$, as required.
\end{proof}
Another canonical method of bounding the spectral gap from below is \emph{the path method}. 

\begin{claim}[The path method, see \cite{LPW2}*{Corollary 13.21}]\label{lem:path:method}
	Let $P$ be a transition matrix of a Markov chain on a finite state space $\Omega$ with stationary distribution $\pi$ and spectral gap $\gamma$. For every $x,y\in \Omega$, denote $Q(x,y) = \pi(x)P(x,y)$. Let $\{\varphi_{x,y}\}_{x,y\in \Omega}$ be a collection of paths in $\Omega$ from $x$ to $y$ such that every path $\varphi_{x,y}$ is a path from $x$ to $y$ with $Q(e) > 0$ for every $e\in \varphi_{x,y}$. Denote 
	\begin{equation*}
	B := \max_e \frac{1}{Q(e)}\sum_{\varphi_{x,y}\ni e}\pi(x)\pi(y)|\varphi_{x,y}|.
	\end{equation*}
	Then, $\gamma\geq B^{-1}$.
\end{claim}

\begin{proof}[Proof of \cref{lem:dec:spec:gap}]
	When $k=1$, this is trivial and the spectral gap of $G$ is $a$. We assume henceforth that $k\geq 2$ and we consider the projection and restriction chains associated with the decomposition of $V$ to sets of $\deco$ as described earlier in this section.
	Let $i\in[k]$ and consider the restriction chain associated with $V_i$, a set of $\deco$. 
	By our assumption, we have that the spectral gap of $G[V_i]$ is at least $a$. 
	However, the restriction walk on $V_i$ is different from the simple random walk on $G[V_i]$, as it is obtained from it by adding self loops for every edge that exits $V_i$. 
	For every $v\in V$ we denote by $p_v$ the probability to move from $v$ to itself in the restriction chain. 
	Since every $v\in V_i$ has $\deg(v, V_i) \geq b$, we have that $p_v \leq (1-\frac{b}{n})$. We thus have by \cref{cl:weird:laziness} that $\gamma_i$, the spectral gap of the restriction chain, satisfies $\gamma_i \geq \frac{ab}{n}$.
	
	We turn to the projection chain, which is a Markov chain on the state space $[k]$ with stationary distribution $\overline{\pi}$ and transition matrix $\overline{P}$ as in \eqref{eq:proj:chain}. We will use \cref{lem:path:method} to bound the spectral gap associated with it from below. For every edge $e = (l,m) \in H(\deco,c)$, we denote $Q(e) = \overline{\pi}(l)\overline{P}(l,m)$. We denote by $P$ the transition matrix of the original simple random walk. By the definition in \eqref{eq:proj:chain}, we have
	\begin{equation*}
	Q(e) = \overline{\pi}(l)\cdot \frac{1}{\overline{\pi}(l)}\sum_{u\in V_l, v\in V_m} \pi(u)P(u,v) = \sum_{u\in V_l} \frac{|E(u,V_m)|}{2|E|} \geq \frac{|E(V_l,V_m)|}{n^2} \geq \frac{c}{n^2}.
	\end{equation*}
	Since the graph $H(\deco,c)$ is connected, for every $x,y\in[k]$ we can choose a path $\varphi_{x,y}$ connecting $x$ and $y$ such that every edge in this path has $Q(e) \geq c/n^2$. We choose such paths for every $x,y$ arbitrarily. We obtain that for every edge $e$ which belongs to any path in this choice  $(\varphi_{x,y})_{x,y\in[k]}$ we have 
	\begin{equation*}
	\frac{1}{Q(e)}\sum_{\gamma_{x,y}\ni e}\overline{\pi}(x)\overline{\pi}(y)|\gamma_{x,y}| \leq \frac{kn^2}{c}\sum_{\gamma_{x,y}\ni e}\overline{\pi}(x)\overline{\pi}(y) \leq \frac{kn^2}{c}.
	\end{equation*}
	Hence, by \cref{lem:path:method}, the spectral gap of the projection chain $\overline{\gamma}$ is at least $c/kn^2$.
	Using \cref{thm:jsvt}, we conclude
	\begin{equation*}
	\gamma \geq \min_{i\in[k]} \frac{\overline{\gamma}\gamma_i}{6} \geq \frac{abc}{6kn^3}.\qedhere
	\end{equation*}
	
\end{proof}

\subsection{Primary decomposition}

\begin{definition}\label{def:primary} Let $G=(V,E)$ be a graph on $n$ vertices with minimal degree at least $\delta n$. A \defn{$\delta$-primary} decomposition of $G$ is a partition of its vertices $V = V_1 \sqcup \ldots \sqcup V_k$ which has the following properties.
	
	\begin{enumerate}
		\item The number of sets in the decomposition, denoted by $k$, satisfies $k \leq 2/\delta$.
		\item For every $i\in[k]$, we have $|V_i| \geq \frac{\delta n}{2}$.
		\item For every $i\in[k]$ and each $v\in V_i$ we have that $\deg(v,V_i) \geq \frac{\delta^4 n}{40}$. 
		\item For every $i\in[k]$, the spectral gap of $G[V_i]$ is at least $\frac{\delta^{10}}{2^{22}}$. 
	\end{enumerate}
\end{definition}

\begin{lemma}\label{cl:prim:deco}
	Let $G=(V,E)$ be a simple graph on $n$ vertices with minimal degree at least $\delta n$. Then, there exists a $\delta$-primary decomposition of $G$.	
\end{lemma}
\begin{proof}
	We build the decomposition inductively each time refining the partition of $V$. We begin with the trivial partition $\{V\}$. 
	At each step, if there is a subset $W$ in the partition that can be further partitioned $W=W_1 \sqcup W_2$ such that $|E(W_1,W_2)| \leq \frac{\delta^3}{20} |W_1||W_2|$, then we refine the partition by replacing $W$ with $W_1,W_2$. We call the edges $E(W_1,W_2)$ in each such refinement step \defn{negligible}. The choice of $W$ and $W_1, W_2$ is not necessarily unique and at each step we choose arbitrarily among all possibilities. Since the graph is finite this process must stop and we denote the final decomposition by $V = U_1 \sqcup \ldots \sqcup U_\ell$, where edges between any pair $U_i$ and $U_j$ are negligible. The sum of $|W_1||W_2|$, described above, over each of the $\ell$ refinement steps is no more than the cardinality of pairs of vertices, hence, the number of negligible edges is at most $\frac{\delta^3}{20}{n \choose 2}$.

	We call a vertex \defn{bad} if the number of negligible edges touching it is larger than $\delta n / 2$. Our bound on the number of negligible edges implies that there are no more than $\delta^2n/10$ bad vertices. 
	Every vertex which is not bad is called \defn{good}. If a set $U_i$ in the partition contains a good vertex we call it a \defn{good set}, otherwise, a \defn{bad set}. Since the minimal degree in the graph is larger than $\delta n$, every good vertex $v$touches at least $\delta n / 2$ edges that are not negligible; the corresponding neighbors must be in the same set of the partition as $v$. Hence each good set is of size at least $\delta n / 2$ and so their number is at most  $2/\delta$. We call bad vertices belonging to bad sets \defn{evil}. To obtain our primary decomposition, we remove all bad sets from the partition and redistribute the evil vertices among the good sets as follows. Assume without loss of generality that the good sets of the partition are $U_1, \ldots, U_k$ where $k\leq \ell$. 
	Let $v \in V \setminus \cup_{i=1}^k U_i$ be an evil vertex. Since the number of good neighbors of $v$ is at least $\delta n - \delta^2n/10$ and $k \leq 2/\delta$, there exists some $i\in[k]$ for which $d(v, U_i) \geq \delta^2 n / 3$. We add $v$ to one such set chosen arbitrarily. 
	
	We denote the resulting decomposition by $V = V_1 \sqcup \ldots \sqcup V_k$ (with $U_i \subset V_i$ for all $i\in[k]$) and argue that it satisfies the desired conditions of \cref{def:primary}. Conditions (1) and (2) are immediate. To see that condition $(3)$ is satisfied, let $i\in[k]$ and let $v\in V_i$. If $v$ is good, then $\deg(v,V_i) \geq \delta n/2$. If $v$ is bad but not evil, then $|E(\{v\}, U_i)| \geq \frac{\delta^3}{20} |U_i| \geq \delta^4 n / 40$ since otherwise we would have partitioned $U_i$ to $\{v\}$ and $U_i\setminus\{v\}$. If $v$ is evil, then it was added to $V_i$ since $d(v,U_i)\geq \delta^2 n / 3$.
	
	It remains to prove condition (4). Due to Cheeger's inequality (see \cref{eq:cheeger}), it is enough to show that for every $i\in[k]$
	\begin{equation}\label{eq:cheeg:to:prove}
	\Phi(G[V_i]) \geq \frac{\delta^5}{1200}.
	\end{equation}
	Fix $i\in [k]$. We slightly abuse notation and write $\Vol$ and $\pi$ for the volume and stationary measures on $G[V_i]$ respectively, that is, for any $S\subset V_i$ we have $\Vol(S)=\sum _{s \in S} \deg(s, V_i)$ and $\pi(S)=\Vol(S)/\Vol(V_i)$. Let $X \subset V_i$ be a subset with $\pi(X) \leq 1/2$. If at least half of the vertices of $X$ are evil, then its size is at most twice the number of bad vertices, i.e. $|X| \leq \delta^2 n / 5$. Each evil vertex of $X$ has at least $\delta^2n/3 - \delta^2n/5$ of its neighbors in $V_i$ outside $X$. Hence,
	\begin{equation*}
	\frac{|\partial X|}{\vol(X)} \geq \frac{\frac{2\delta^2n}{15}\cdot\frac{|X|}{2}}{n|X|} = \frac{\delta^2}{15}.
	\end{equation*} Suppose otherwise that at least half of the vertices of $X$ are non-evil. Denote the set of non-evil vertices of $X$ by $R$ and let $T$ be the other non-evil vertices of $V_i$. Note that $R\sqcup T = U_i$. The number of edges between them is at least $\frac{\delta^3}{20} |R||T|$, since otherwise $U_i$ would have been partitioned further. Thus,
	\begin{equation}\label{eq:est:of:t}
	\frac{|\partial X|}{\Vol(X)} \geq \frac{|E(R,T)|}{\Vol(X)} \geq \frac{\frac{\delta^3}{20}|R||T|}{\Vol(X)} \geq \frac{\frac{\delta^3}{20} |T| |X|}{2n|X|} \geq \frac{\delta^3|T|}{40n}.
	\end{equation}
	It remains to lower bound $|T|$. Denote by $G_i$ and $B_i$ the sets of good and bad vertices of $V_i$, respectively. We have that $|G_i| \geq \delta n/2 - \delta^2 n/10$. Each vertex $v\in G_i$ has $\deg(v,V_i)\geq \delta n/2$ hence $\Vol(G_i) \geq \delta^2 n^2/5$. On the other hand, since the total number of bad vertices is at most $\delta^2 n/10$ we have $\Vol(B_i)\leq \delta^2 n^2/10$. We deduce that $\pi(G_i)\geq 2/3$. Since $\pi(V_i \setminus X) \geq 1/2$, we have that $\pi(T) \geq \pi(G_i \cap (V_i \setminus X)) \geq 1/6$. We bound $\Vol(T)\leq |T|n$ and $\Vol(V_i)\geq\Vol(G_i) \geq \delta^2 n^2/5$ which with the last estimate gives $|T|\geq \delta^2 n /30$. We plug this into \eqref{eq:est:of:t} to obtain that $\frac{|\partial X|}{\Vol(X)} \geq \frac{\delta^5}{1200}$, as required.
\end{proof}

\subsection{The Cheeger constant and spectral gap are comparable on graphs with linear degree}\label{sec:cheegereverything}

\begin{proof}[Proof of \cref{thm:cheeger:is:everything}]
	For brevity we denote $\Phi(G) = r$.
	By \cref{cl:prim:deco}, there exists a $\delta$-primary decomposition $\deco$ of $G$, also denoted by $V=V_1 \sqcup \ldots \sqcup V_k$.
	We claim that the graph $H(\deco, r\delta^2 n^2 / 2k^2)$ is connected. Indeed, assume to the contrary that it is not connected and let $S\subseteq [k]$, $S \neq [k]$, be a connected component in this graph. Let $V_S := \cup_{i\in S}V_i$ be the set corresponding to $S$ in $G$. We may assume that $\pi(V_S) \leq 1/2$ (otherwise, we will take one of the connected components of $V_{[k] \setminus S}$). Since $\deco$ is $\delta$-primary and $\Phi(G) = r$, we have that
	\begin{equation*}
	|\partial V_S| \geq r\vol(V_S) \geq r \delta^2 n^2 / 2. 
	\end{equation*}
	Hence, we can find a component $V_j \subseteq (V\setminus V_S)$ and a component $V_i \subseteq V_S$ with $|E(V_i,V_j)| \geq r \delta^2 n^2 / 2k^2$, contradicting the assumption that $S$ is a connected component of $H(\deco, r\delta^2 n^2 / 2k^2)$. We can therefore apply \cref{lem:dec:spec:gap} and obtain that \cref{thm:cheeger:is:everything} holds with the constant $\frac{\delta^{19}}{2^{34}}$ (note that $\delta \leq 1$ so the minimum in the conclusion of \cref{lem:dec:spec:gap} is attained in the second item).
\end{proof}

\subsection{Coarsening}

Given two partitions $\deco$ and $\deco'$, we say that $\deco$ is a \defn{coarsening} of $\deco'$ if every set in $\deco$ is a union of sets in $\deco'$.  
Suppose that $G=(V,E)$ has minimal degree at least $\delta n$. Let $\deco$ be a partition denoted by $V=V_1 \sqcup \ldots \sqcup V_k$, and let $\deco'$ be a partition $V=V'_1\sqcup\ldots\sqcup V'_\ell$ such that $\deco$ is a coarsening of $\deco'$. For each $i\in[k]$ we write $\deco_i$ for the partition of $V_i$ into sets of $\deco'$. 

\begin{definition} \label{def:good:coarse} For $\eps,\alpha\in(0,1)$ and $\beta>0$ we say that $\deco$ is an \defn{$(\eps,\alpha,\beta)$-good coarsening} of $\deco'$, if there exists some $\theta \in \left[\eps\left(\frac{\eps\alpha}{\ell^2}\right)^{2^\ell},\eps\right]$ for which the following conditions are satisfied.
	\begin{enumerate}
		\item  For every $i\in[k]$ we have that $H(\deco_i,\theta\beta)$ is connected ($H$ is defined in \cref{def:hp}).
		\item For every $i\in[k]$ we have $|E(V_i,V\setminus V_i)| \leq \theta^2 \beta\alpha$.
	\end{enumerate}
\end{definition}

\begin{lemma}\label{lem:dec:coarse}
	For any $\eps,\alpha\in (0,1)$ and any $\beta>0$, if $G=(V,E)$ is a finite graph and $\deco'$ is a partition of $V$, then there exists an $(\eps,\alpha,\beta)$-good coarsening of $\deco'$.
\end{lemma}

\begin{proof}
	Let $\eps>0$. We will construct $\deco$, an $(\eps,\alpha,\beta)$-good coarsening of $\deco'$, which will satisfy conditions (1) and (2) of \cref{def:good:coarse} with some parameter $\theta$. At first, if $\deco'$ satisfies condition (2) with $\eps$ playing the role of $\theta$, then $\deco'$ is an $(\eps,\alpha,\beta)$ good coarsening of itself. Otherwise, we build recursively a finite sequence of length at most $\ell$ of coarsenings $\langle\deco_i\rangle$ of $\deco'$ and parameters $\langle\theta_i\rangle$, such that $\deco_i$ satisfies condition (1) with $\theta_i$. We set $\deco_1 := \deco'$ and $\theta_1 = \eps$.
	
	At step $m>1$, we are given with $\deco_{m-1}$ and $\theta_{m-1}$ such that $\deco_{m-1}$ satisfies condition (1) with parameter $\theta_{m-1}$. If $\deco_{m-1}$ also satisfies condition (2) with $\theta_{m-1}$, then $\deco_{m-1}$ is an $(\eps,\alpha,\beta)$ good coarsening and we halt the process, denoting $\deco := \deco_{m-1}$ and $\theta := \theta_{m-1}$. Otherwise, there exists a set $U$ in the partition $\deco_{m-1}$ which has $|E(U,V\setminus U)| \geq \theta_{m-1}^2\alpha\beta$. Since there are $\ell$ sets in $\deco'$, there exists at least one pair of sets $W_1 \subseteq U$ and $W_2 \subseteq V\setminus U$, both are sets of the partition $\deco'$, such that $|E(W_1,W_2)| \geq \theta_{m-1}^2\alpha\beta / \ell^2$. We denote 
	\begin{equation}\label{eq:theta:relation}
	\theta_m = \frac{\alpha}{\ell^2}\theta_{m-1}^2
	\end{equation}  
	and form $\deco_m$ by replacing $U$ and the set containing $W_2$ in $\deco_{m-1}$ with their union. We note that since we assumed that $\deco_{m-1}$ is a coarsening of $\deco'$ and satisfies condition (1) with $\theta_{m-1}$, then $\deco_m$ is also a coarsening of $\deco'$ which satisfies condition (1) with $\theta_m$. 
	
	Eventually, since the number of sets in $\deco'$ is $\ell$, this process halts within at most $\ell$ steps and we obtain an $(\eps,\alpha,\beta)$ good coarsening $\deco$ and $\theta$, a parameter satisfying $\theta \geq \theta_\ell$, with which conditions (1) and (2) are satisfied. Solving \eqref{eq:theta:relation} with initial condition $\theta_1 = \eps$, we obtain
	\begin{equation*}
	\theta_m = \eps\left(\frac{\eps\alpha}{\ell^2}\right)^{2^{m-1}-1}.
	\end{equation*} Therefore, we have
	\begin{equation*}
	\theta \geq \theta_{\ell} \geq \eps\left(\frac{\eps\alpha}{\ell^2}\right)^{2^{\ell-1}-1} \geq \eps\left(\frac{\eps\alpha}{\ell^2}\right)^{2^\ell},
	\end{equation*} 
	as required.
\end{proof}

\subsection{Proof of \cref{lem:main:deco}}
To prove \cref{lem:main:deco} we will show that with the right choice of $\alpha$ an $(\eps,\alpha,\beta)$-good coarsening of a $\delta$-primary decomposition is in fact a $(\eps,\delta,\beta)$-good decomposition.
\begin{proof}[Proof of \cref{lem:main:deco}]
	Let $G=(V,E)$ be a graph on $n$ vertices with minimal degree at least $\delta n$, let $\eps>0$ and let $0 < \beta\leq 240n^2/(\eps\delta^4)$. By \cref{cl:prim:deco}, there exists a $\delta$-primary decomposition of $V$. We denote this decomposition by $\deco'$. By \cref{lem:dec:coarse}, we obtain that there exists an $(\eps,\eps^9,\beta)$-good coarsening of $\deco'$, denoted by $\deco$. We also denote this coarsening explicitly by $V=V_1 \sqcup \ldots \sqcup V_k$ and we let $\theta$ be the parameter from \cref{lem:dec:coarse} to which the coarsening corresponds.
	
	We claim that $\deco$ is indeed an $(\eps,\delta,\beta)$-good decomposition satisfying conditions (3) and (5) of \cref{def:good:deco} with this $\theta$. We first note that for $\eps>0$ small enough and by the properties of a good coarsening
	\begin{equation*}
	\eps^{11\cdot 2^{2/\delta}} \leq \eps\left(\frac{\eps^{10}\delta^2}{4}\right)^{2^{2/\delta}} \leq \theta \leq \eps.
	\end{equation*}Conditions (1), (2) and (4) of \cref{def:good:deco} are immediate for every coarsening of a $\delta$-primary decomposition. Condition (5) is satisfied by condition (2) of the coarsening in \cref{def:good:coarse}. We are then left with verifying that condition (3) of \cref{def:good:deco} holds.
	To this end, we let $V_i$ be some set in $\deco$ and $V_i = V_{i,1}\sqcup\ldots\sqcup V_{i,\ell_i}$ be its partition to $\ell_i$ sets of $\deco'$ which we denote by $\deco_i$. 
	Note that for every $j\in[\ell_i]$ and for every $v\in V_{i,j}$ we have $\deg(v,V_{i,j}) \geq \delta^4n/40$. Also, by condition (4) of the primary decomposition, \cref{def:primary}, we have that the spectral gap corresponding to $G[V_{i,j}]$ is at least $\delta^{10}/2^{22}$. Finally, by the properties of the coarsening, the graph $H(\deco_i, \theta \beta)$ is connected. Hence, denoting $\gamma(G[V_i])$ for the spectral gap of $G[V_i]$ and using \cref{lem:dec:spec:gap} we get that
	\begin{equation}\label{eq:sp:gap:min:bound}
	\gamma(G[V_i]) \geq \min\left\{\frac{\delta^{10}}{2^{22}},\quad \frac{\delta^4 n}{40}\cdot\frac{\delta^{10}}{2^{22}}\cdot\theta \beta\cdot\frac{1}{6\ell_in^3}\right\}. 
	\end{equation}
	Since $\theta \leq \eps$ and $\beta\leq 240n^2/(\eps\delta^4)$ and $\ell_i \geq 1$ we learn that the minimum above is attained in the second term. As $\deco$ is a coarsening of a $\delta$-primary decomposition, we have that $\ell_i\leq 2/\delta$, hence,
	\begin{equation*}
	\gamma(G[V_i]) \geq \frac{\delta^4 n}{40}\cdot\frac{\delta^{10}}{2^{22}}\cdot\theta \beta\cdot\frac{1}{6\ell_in^3} \geq	\frac{\delta^{15}\theta\beta}{2^{31}n^2}
	\end{equation*}
and we deduce that condition (3) of \cref{def:good:deco} holds. 
\end{proof}

\section{Random walks and Wilson's algorithm on decomposed graphs}\label{sec:rwdeco}

In the rest of this paper on we take $\beta=n^{3/2}$ in the decomposition of \cref{sec:dec}. The main goal of this section is to prove the following estimate. 

\begin{theorem}\label{lem:paths:in:deco}
	For any $\delta\in(0,1]$ there exists $C = C(\delta)<\infty$ such that the following holds. 
	Let $G=(V,E)$ be a connected simple graph on $n$ vertices with minimal degree at least $\delta n$ and $\eps>0$. Denote by $V=V_1 \sqcup \ldots \sqcup V_k$ an $(\eps,\delta,n^{1.5})$-good decomposition of $G$ with parameter $\theta$ (as guaranteed to exist by \cref{lem:main:deco}). Then for every $i\in[k]$ there are two vertices $v_1^i, v_2^i \in V_i$ such that if $\cT$ is a uniform spanning tree of $G$ and $\varphi$ is the unique path between $v_1^i$ and $ v_2^i$ in $\cT$, then
	\begin{equation*}
	\pr\left(\eps^8\theta\sqrt{n} \leq |\varphi| \leq \frac{\sqrt{n}}{\theta\eps^{8}} \pand \varphi \subseteq V_i	 \right) \geq 1-C\eps^2.
	\end{equation*}
\end{theorem}

In \cref{subsec:prelimrw} we prove a couple of preliminary useful random walk estimates on graphs with linear minimal degree that do not involve the decomposition. In \cref{sec:rw:stay:in:comp}, we show that a random walk on $G$ stays inside one set of the decomposition for at least $\sqrt{n}$ steps with high probability; since the spectral gap of each set in the decomposition is at least of order $n^{-1/2}$ and the induced graph on the set has linear minimal degree, the random walk is mixed in this set (even though the mixing time of $G$ may be much larger than $\sqrt{n}$). We use this estimate in \cref{sec:rw} to prove the aforementioned \cref{lem:paths:in:deco}.

\subsection{Preliminary random walk estimates}\label{subsec:prelimrw}

It is a classical fact that the mixing time of the random walk on a connected graph $G$ is always $O(\gamma^{-1} \log n)$ where $n$ is the number of vertices and $\gamma=\gamma(G)$ is the spectral gap, see for instance \cite[Theorem 12.4]{LPW2}. This estimate is sharp as is seen on bounded degree expander graphs where the gap is $\Omega(1)$ but at least $\Omega(\log n)$ steps are needed for the walker to be able to reach the majority of the graph. However, when the minimal degree is linear, after a single step the location is already spread on a set of linear size and this estimate can be improved.

\begin{lemma}\label{lem:mix:spec}
	For any $\delta\in(0,1]$ there exists a constant $C = C(\delta)<\infty$ such that the following holds. For any simple graph $G$ on $n$ vertices with minimal degree at least $\delta n$ and any $\eps\in(0,1)$ we have
	\begin{equation*}
		\tmix^G(\eps) \leq C\log(1/\eps)\left(\gamma^{-1}(G) + \log(n)\right).
	\end{equation*} 
\end{lemma}

\begin{proof}
Let $P$ be the transition matrix of the lazy random walk on $G$. Recall that $P$ is a self-adjoint operator $P:L^2(\pi)\to L^2(\pi)$ where $\pi(v)=\deg_G(v)/2|E(G)|$ is the stationary distribution. We denote the eigenvalues of $P$ by $1 = \lambda_1 \geq \ldots \geq \lambda_n\geq 0$ and $\gamma(P) = 1-\lambda_2$ and by $\ind$ the all $1$ vector which is the eigenvalue corresponding to the eigenvalue $1$.  
Let $\mu$ be any probability measure on $V$ and write $f$ for the vector $f(v)=\mu(v)/\pi(v)$. We have that $f-\ind$ is orthogonal to $\ind$ and for any integer $t\geq 1$ we have that $\pi(v) Pf(v) = \pr_\mu(X_t = \cdot)$ so in particular $P^t f - \ind$ is orthogonal to $\ind.$ Hence
\begin{equation}\label{eq:mixing:time:second:ev}
		\left\| P^t f - \ind  \right\|_2 = \left\| P^t\left(f - \ind\right) \right\|_2  
		\leq \lambda_2^t\left\|f - \ind \right\|_2 .	
\end{equation}
We rewrite this as 
	\begin{equation*}
		\left\| \frac{\pr_\mu(X_t = \cdot)}{\pi(\cdot)} - \ind \right\|_2 \leq \lambda_2^t \left\|f - \ind \right\|_2.
	\end{equation*} 
	We now claim that for any $v\in V$ and every $u\in V$ we have that
	\begin{equation*}
		\pr_v(X_{\ceil{\log_2(n)}} = u) \leq \frac{2}{\delta n}.
	\end{equation*}
	Indeed, if the random walker made a non-lazy step at some time in $\{1, \ldots, \ceil{\log_2(n)}\}$, then the probability to be at any vertex $u$ at time $\ceil{\log_2(n)}$ is bounded by $1/(\delta n)$. On the other hand, the probability of staying put $\ceil{\log_2(n)}$ steps is at most ${1 \over n}$. Since ${\delta \over n} \leq \pi(\cdot) \leq {1 \over \delta n}$ it follows that
	\begin{equation}\label{eq:boun:on:logn}
		\left\| \frac{\pr_v(X_{\ceil{\log_2(n)}} = \cdot)}{\pi(\cdot)} - \ind \right\|_2^2  \leq \frac{2}{\delta^2}.
	\end{equation}
	Using \eqref{eq:mixing:time:second:ev} and \eqref{eq:boun:on:logn}, for $t = \ceil{\log_2(n)} + \log(\sqrt{2}/\eps \delta)\gamma^{-1}$ and every $v\in V$ we have
	\begin{equation*}
		\left\| \frac{\pr_v(X_t = \cdot)}{\pi(\cdot)} - \ind \right\|_2 \leq \lambda_2^{t-\ceil{\log_2(n)}}\left\| \frac{\pr_v(X_{\ceil{\log_2(n)}} = \cdot)}{\pi(\cdot)} - 1 \right\|_2 \leq \frac{\sqrt{2}}{\delta}(1-\gamma)^{t-\ceil{\log_2(n)}} \leq \eps. \end{equation*} By \cite{LPW2}*{Lemma 12.18} we have that
	\begin{equation*}
		2\|\vect{p}_t(v,\cdot) - \pi(\cdot)\|_{\dtv} \leq \left\| \frac{\pr_v(X_t = \cdot)}{\pi(\cdot)} - 1 \right\|_2 \leq \eps \, ,
	\end{equation*}
concluding our proof.
\end{proof}

\begin{claim}\label{cl:hit:large:set}
	For any $\delta>0$ there exists $c = c(\delta)>0$ such that the following holds. For any $\eps\in(0,c)$, any simple graph $G = (V,E)$ on $n\geq \eps^{-2}$ vertices with minimal degree at least $\delta n$, any $U\subset V$ with $|U| \geq \eps\sqrt{n}$ and any vertex $v\in G$ 
	\begin{equation*}
		\pr_v(X[0,2(\tmix^G(\eps/2)+\floor{\sqrt{n}})] \cap U = \es) \leq 1-\frac{\eps\delta}{4} \, ,
	\end{equation*}
	where $X$ is the simple random walk on $G$.
\end{claim}

\begin{remark}\label{rmrk:tmix} We emphasize a potentially confusing point: $\tmix$ is defined in terms of the lazy random walk, but in this claim, as well as the rest of this paper, we study the \emph{non-lazy} random walk running for times depending on $\tmix$. 
\end{remark}

\begin{proof} We prove this for the \emph{lazy} simple random walk and trivially it follows for the usual random walk. Without loss of generality we may assume that $|U|=\ceil{\eps \sqrt{n}}$, otherwise we may take a subset of $U$ of that size. By \cref{eq:exp:mix} we have that $2(\tmix^G(\eps/2) + \floor{\sqrt{n}}) \geq \tmix^G(\eps^2) + \floor{\sqrt{n}}$ so it is then enough to bound from below the probability that $X$ hits $U$ within $\tmix^G(\eps^2) + \floor{\sqrt{n}}$ steps. Recall that 
	\begin{equation}\label{eq:paley:zygmund}
		\pr(Z>0) \geq \frac{\E^2[Z]}{\E[Z^2]} \, ,
	\end{equation}
	for any non-negative random variable $Z$. 
	Let $Y$ be a lazy random walk starting from the stationary distribution and define $Z = |\{t\in[0,\floor{\sqrt{n}}] \mid Y_t\in U\}|$. We will use \eqref{eq:paley:zygmund} to bound $\pr(Z>0)$ from below. Since $|U|\geq\eps\sqrt{n}$ and the minimal degree is $\delta n$, we have that $\pi(U) \geq \eps \delta n^{-1/2}$ and so $\E[Z] = (\floor{\sqrt{n}}+1) \pi(U) \geq \eps\delta$. 
	
	To bound $\E[Z^2]$, let $t<r$ be two positive integers. If $Y_t \in U$, there is a probability of $2^{-(r-t)}$ that the walker made $r-t$ lazy steps and then $Y_r = Y_t$. Else, since the minimal degree is at least $\delta n$, the probability that $Y_r \in U$ is bounded by $|U|/(\delta n)$. Therefore, 
	\begin{equation*}
		\pr(Y_t \in U, Y_r \in U) \leq \pr(Y_t \in U)\cdot\left(\frac{|U|}{\delta n} + \frac{1}{2^{r-t}}\right) = \pi(U) \cdot\left(\frac{|U|}{\delta n} + \frac{1}{2^{r-t}}\right).
	\end{equation*} 
	Hence, by summing over all $t\leq r\leq\floor{\sqrt{n}}$
	\begin{align*}
		\E[Z^2] &= \sum_{t=0}^{\floor{\sqrt{n}}} \pr(Y_t\in U) + 2\sum_{t=0}^{\floor{\sqrt{n}}}\sum_{r=t+1}^{\floor{\sqrt{n}}} \pr(Y_t\in U, Y_r\in U) 
		\\&\leq \E[Z] + 2\binom{\floor{\sqrt{n}}+1}{2}\frac{\pi(U)|U|}{\delta n} + 2\sum_{t=0}^{\floor{\sqrt{n}}}\pi(U)\sum_{r=t+1}^{\floor{\sqrt{n}}}\frac{1}{2^{r-t}}.
	\end{align*}
	Since $\E[Z] = (\floor{\sqrt{n}}+1)\pi(U)$ we upper bound the last term of the right-hand side by $2\E[Z]$. For the middle term we write
	\begin{align*}
		2\binom{\floor{\sqrt{n}}+1}{2}\frac{\pi(U)|U|}{\delta n} \leq (\floor{\sqrt{n}} + 1)\cdot\frac{\pi(U)|U|}{\delta \sqrt{n}}\leq \E[Z]\left(\frac{\eps\sqrt{n} +1}{\delta\sqrt{n}}\right)\leq \frac{1}{2}\E[Z],  
	\end{align*}
	where the last inequality holds for $\eps$ small enough. Therefore, for $\eps$ small enough we have that $\E[Z^2] \leq \frac{7}{2}\E[Z]$ and thus by \eqref{eq:paley:zygmund}
	\begin{equation*}
		\pr(Z>0) \geq \frac{2\E[Z]}{7} \geq \frac{2\eps\delta}{7}.
	\end{equation*}
	By the definition of $\tmix^G(\eps^2)$, if $X$ is a random walk starting from some $v\in V$, we have that $\dtv(X_{\tmix(\eps^2)},\pi) \leq \eps^2$. Therefore, we can couple the walk $X$ starting from time $\tmix^G(\eps^2)$ with an independent random walk $Y$ starting from the stationary distribution such that the walks coincide with probability larger than $1-\eps^2$. We thus obtain that for $\eps$ small enough
	\begin{equation*}
		\pr_v(X[\tmix^G(\eps^2), \tmix^G(\eps^2)+\floor{\sqrt{n}}]\cap U = \es) \geq \eps\delta/3.5 - \eps^2 \geq \eps\delta/4.\qedhere
	\end{equation*}
\end{proof}

\subsection{Random walks of length $\sqrt{n}$ stay in the same set of the decomposition}\label{sec:rw:stay:in:comp}

We now show that with high probability the random walker on a graph with linear minimal degree will stay in the same set of its $(\eps,\delta,n^{1.5})$-good decomposition that it walked to in its first step.

\begin{lemma}\label{cl:stay:comp:good:deco}
	Let $\delta\in(0,1]$, $\eps>0$ and $G=(V,E)$ be a simple graph on $n$ vertices with minimal degree at least $\delta n$. Also let $V=V_1 \sqcup \ldots \sqcup V_k$ be an $(\eps,\delta,n^{1.5})$-good decomposition of $G$ with parameter $\theta$ (as guaranteed to exist by \cref{lem:main:deco}). Then for any $C>0$ and any $i\in [k]$
	\begin{equation}\label{eq:stay:in:comp}
	\pr_v\left(\exists t\in\left[1,C\sqrt{n}\right]: X_t \in V_i, X_{t+1} \notin V_i\right) \leq \frac{C\theta^2\eps^9}{\delta^2} \, ,
	\end{equation}
	where $X$ is the simple random walk on $G$. Furthermore, for any $i\in [k]$ there exists a set $V_i'\subseteq V_i$ satisfying $|V_i'| \geq \delta^4 n / 80$ such that for every $v\in V_i'$  
	\begin{equation}\label{eq:good:vertex:stays:in:set}
		\pr_v\left(X\left[0,C\sqrt{n}\right] \subseteq V_i\right) \geq  1-\frac{80C\theta^2\eps^9}{\delta^6}.
	\end{equation}
\end{lemma}

\begin{proof}
Let $i\in[k]$ and let $\theta$ be the parameter from the $(\eps,\delta,n^{1.5})$-good decomposition $\deco$. For every $t\geq 1$, we have 
	\begin{equation*}
	\pr(X_t \in V_i, X_{t+1}\not \in V_i) \leq \sum_{w\in V_i} \pr(X_t = w)\cdot \vect{p}(w,V_i^c) \leq \sum_{w\in V_i}\frac{1}{\delta n}\frac{|E(w,V_i^c)|}{\delta n} \leq \frac{|E(V_i,V_i^c)|}{\delta^2 n^2} \, .
	\end{equation*}
Hence by taking the union over $t\in[1,C\sqrt{n}]$ and using condition (5) of a $(\eps,\delta,n^{1.5})$-good decomposition (\cref{def:good:deco}) we immediately obtain \eqref{eq:stay:in:comp}.
To prove \eqref{eq:good:vertex:stays:in:set} let $C>0$ and fix some $v\in V_i$. By condition (4) of \cref{def:good:deco}, we have that $\deg(v,V_i) \geq \delta^4n/40$. By \eqref{eq:stay:in:comp} we have
		\begin{equation*}
		\pr_v(X_1\in V_i \pand X\left[1,C\sqrt{n} + 1\right] \not\subseteq V_i) \leq \frac{C\theta^2\eps^9}{\delta^2}.
		\end{equation*}
Yet on the other hand,
		\begin{align*}
		&\pr_v(X_1 \in V_i \pand X[1,C\sqrt{n}] \not\subseteq V_i) \geq \frac{1}{n} \sum_{u\sim v, u\in V_i} \pr_u\left(X[0,C\sqrt{n}] \not\subseteq V_i\right).
		\end{align*}
Thus the number of $u \in V_i$ with $u\sim v$ satisfying 
\begin{equation*}
		\pr_u\left(X[0,C\sqrt{n}] \not\subseteq V_i\right) \geq \frac{80C\theta^2\eps^9}{\delta^6}.
\end{equation*}
cannot be larger than $\delta^4n/80$. Since $\deg(v,V_i) \geq \delta^4n/40$ we conclude the proof of \eqref{eq:good:vertex:stays:in:set}.
\end{proof}

\subsection{$\LERW$s and Wilson's Algorithm on the decomposed graph}\label{sec:rw}

We now proceed to the proof of \cref{{lem:paths:in:deco}}. In the rest of this section we assume that $\delta\in(0,1]$ and $\eps>0$ are given, that $G = (V,E)$ is a connected simple graph on $n$ vertices with minimal degree $\delta n$ and that $V=V_1\sqcup \ldots \sqcup V_k$ is an $(\eps,\delta,n^{1.5})$-good decomposition with parameter $\theta$ as guaranteed to exist by \cref{lem:main:deco}. Lastly, note that if $\sqrt{n} \leq 1/ (\theta \eps^8)$, then \cref{{lem:paths:in:deco}} is trivial, so we assume the contrary. 

As described in \cref{subsec:prelim}, the $\UST$ path between two vertices of $G$ is distributed as the $\LERW$ between them. A recurring problem in analyzing the $\UST$ is that the random walk path between two vertices may be much longer than its loop-erasure, meaning that most of the random walk path is erased during the loop erasure. To overcome this obstacle we use the following idea which goes back to Wilson \cite{Wil96} and was used extensively by Peres and Revelle \cite{PR04+}. Let $G^\rho$ be the network obtained from $G=(V,E)$ by adding a vertex $\rho$, connecting it to each $v\in V$ and assigning edge weights
\begin{equation*}
w(v,\rho) = \frac{\theta\eps^4\deg_G(v)}{\sqrt{n}-\theta\eps^4}.
\end{equation*}
for any $v\in V$. An immediate calculation shows that with these edge weights the probability that the random walk starting from any $v\in V$ moves to $\rho$ in the first step is $ \theta \eps^{4} n^{-1/2}$ and so $\tau_\rho$ is a geometric random variable with expectation $\sqrt{n}/\theta\eps^{4}$. Furthermore, if $X$ is a simple random walk on this network, then conditioned on $\tau_{\rho} = m$, we have that $X[0,m-1]$ is distributed like a simple random walk on $G$ of length $m-1$. 
Thus, in $G^\rho$, the random walk typically takes $\sqrt{n}$ steps to hit $\rho$. It turns out that a positive fraction of such a walk survives the loop erasure with high probability, and is hence easier to analyze. To deduce information about the $\LERW$ in $G$ rather than $G^\rho$ we use \cref{lem:neg:ass:net} stating that $\UST(G)$ stochastically dominates $\UST(G^\rho)\cap E(G)$. 
Hence, if $v_1$ and $v_2$ are two distinct vertices and $\varphi$ and $\varphi'$ are the unique paths between them in $\UST(G)$ and $\UST(G^\rho)$, respectively, then 
\begin{equation}\label{eq:dtv:original:path}
\dtv(\varphi,\varphi') \leq \pr(\rho \in \varphi').
\end{equation}

The proof strategy of \cref{lem:paths:in:deco} is as follows. Fix some $i \in [k]$ and two vertices $v_1,v_2$ in $V_i$ which we will choose according to \eqref{eq:good:vertex:stays:in:set}. We run Wilson's algorithm (see \cref{subsec:prelim}) on $G^\rho$ where the first three vertices in the ordering of the vertices of $G^\rho$ are $(\rho,v_1,v_2)$. We will first show that with high probability the random walk from $v_1$ to $\rho$ stays within $V_i$ except for the last step (\cref{cl:hit:rho:earlier}) and that the length of its loop-erasure is at least $\eps\sqrt{n}$ (\cref{cl:first:lerw}); it is also unlikely to contain $v_2$. Next (\cref{lem:path:builder}) we show that conditioned on this first $\LERW$, with high probability the second $\LERW$ starting at $v_2$ hits the first $\LERW$ in a vertex different than $\rho$ and stays in $V_i$ until that visit. We will also show there that the second $\LERW$ is typically longer than $\eps^8\theta\sqrt{n}$. This gives a lower bound on $|\varphi'|$ and by \eqref{eq:dtv:original:path}  a lower bound for $|\varphi|$ is obtained.

\begin{claim}\label{cl:hit:rho:earlier}
	For any $i\in[k]$ there exists a set $V'_i\subseteq V_i$ with $|V'_i|\geq \delta^4 n / 80$ such that for every $v\in V_i'$, a simple random walk on $G^\rho$ starting from $v$ satisfies
	\begin{equation*}
	\pr_v(\tau_\rho < \tau_{V\setminus V_i}) \geq 1-C\eps^2,
	\end{equation*} 
	for some $C = C(\delta)<\infty$.
\end{claim}

\begin{proof} Apply \cref{cl:stay:comp:good:deco} with $C=1/\theta\eps^6$ to obtain a set $V'_i$ such that for every $v\in V_i'$, if $Y$ is a simple random walk on the graph $G$, then
	\begin{equation}\label{eq:stay in set bound}
	\pr_v(Y[0,\sqrt{n}/\theta\eps^6] \subseteq V_i) \geq 1-\frac{80\theta\eps^3}{\delta^6} \geq 1-C\eps^3.
	\end{equation}
	Also, for a random walk $X$ on $G^\rho$ we have $\E_v \tau_\rho = \sqrt{n}/\theta\eps^{4}$ and thus Markov's inequality gives 
	\begin{equation}\label{eq:avoid the sun bound}
	\pr_v(\tau_\rho \leq \sqrt{n}/\theta\eps^{6})\geq 1-\eps^2
	\end{equation} 
	Let $X$ be a random walk on $G^\rho$. Conditioned on $\tau_\rho$ the random path $X[0,\tau_\rho-1]$ has the distribution of a random walk on $G$. Hence combining \eqref{eq:stay in set bound} and \eqref{eq:avoid the sun bound} yields the desired result.
\end{proof}

\begin{claim}\label{cl:first:lerw}
	For any $i\in[k]$ there exists a set $V'_i\subseteq V_i$ with $|V'_i|\geq \delta^4 n / 80$ such that for every $v\in V_i'$ if $X$ is a simple random walk on $G^\rho$ starting at $v$ and stopped when hitting $\rho$, then 
	\begin{equation*}
	\pr_v(|\LE(X)|  \geq \eps\sqrt{n} \pand \LE(X) \subseteq V_i\cup\{\rho\}) \geq 1-C\eps^2,
	\end{equation*}
	for some $C = C(\delta)<\infty$.
\end{claim}
\begin{proof}
	As in the previous proof, for every $v\in V_i$ we have that $\pr(\tau_{\rho} > \sqrt{n}) \geq 1-\eps^4$. We condition on this event on $\tau_\rho$ and on the first $\tau_\rho - \eps \sqrt{n}$ steps of the random walk. Let us assume first that 
	\begin{equation*}
	|\LE(X[0,\tau_{\rho} - \eps\sqrt{n}))| \geq \eps\sqrt{n}.
	\end{equation*}
	In this case we denote by $U$ the first $\eps \sqrt{n}$ vertices of $\LE(X[0,\tau_{\rho} - \eps\sqrt{n}))$. 
	As explained earlier, under this conditioning the random walk at times $[\tau_{\rho} - \eps\sqrt{n},\tau_{\rho})$ is distributed as a unconditional random walk. Hence, since the minimal degree of $G$ is at least $\delta n$, the probability that $X_t\in U$ for some $t\in [\tau_{\rho} - \eps\sqrt{n},\tau_{\rho}]$ is at most $\eps^2/\delta$. If this does not occur, then the set $U$ survives the loop erasure. It follows that
	\begin{equation*}\label{eq:lerw:prefix:long:total:short}
	\pr\left(	|\LE\left(X[0,\tau_{\rho} - \eps\sqrt{n}] \right)| \geq \eps\sqrt{n} \pand |\LE(X)| \leq \eps\sqrt{n} \right) \leq \eps^4 + \eps^2/\delta \leq C\eps^2 ,
	\end{equation*}
	for some $C=1+\delta^{-1}$. In the second case $|\LE(X[0,\tau_{\rho} - \eps\sqrt{n}))| \leq \eps\sqrt{n}$. In this case the last $\eps\sqrt{n}$ steps of $X$ will survive the loop erasure high probability. Indeed, by Markov's inequality and since the degrees are at least $\delta n$ we deduce that the walk $X[\tau_{\rho} - \eps\sqrt{n},\tau_{\rho}]$ has no loops with probability at least $1-\eps^2/\delta$. By the same reasoning, the walk $X[\tau_{\rho} - \eps\sqrt{n},\tau_{\rho}]$ does not visit  $|\LE(X[0,\tau_{\rho} - \eps\sqrt{n}))|$ with probability at least $1-\eps^2/\delta$. On these two events $\LE(X)$ contains $X[\tau_{\rho} - \eps\sqrt{n},\tau_{\rho}]$. 
	It follows that
	\begin{equation*}\label{eq:lerw:prefix:short:total:short}
		\pr(	|\LE(X[0,\tau_{\rho} - \eps\sqrt{n}))| \leq \eps\sqrt{n} \pand |\LE(X)| \leq \eps\sqrt{n} ) \leq \eps^4 + 2\eps^2/\delta \leq C\eps^2  ,
	\end{equation*}
	for some $C=1+2\delta^{-1}$. Combining the last two inequalities and using \cref{cl:hit:rho:earlier} finishes the proof.
\end{proof}

\begin{lemma}\label{lem:path:builder}
For any $i\in[k]$ there exist two distinct vertices $v_1, v_2 \in V_i$ such that the following holds. Let $X$ be a simple random walk in $G^\rho$ starting at $v_1$ and stopped when hitting $\rho$, and conditioned on $X$, let $Y$ be an independent simple random walk on $G^\rho$ starting at $v_2$ and stopped when hitting $\LE(X)$. Then,
	\begin{equation}\label{eq:second:path:is:good}
	\pr\left(Y_{\tau_{\LE(X)}} \neq \rho \pand |\LE(Y)| \geq \eps^8\theta\sqrt{n} \pand \LE(Y)\subseteq V_i)\right) \geq 1-C\eps^2,
	\end{equation}
	for some $C = C(\delta)<\infty$.
\end{lemma}

\begin{proof}
We apply \cref{cl:first:lerw} to obtain the set $V_i'$ and take $v_1,v_2$ to be any two distinct vertices of $V_i'$. Let $Z$ be simple random walk on $G[V_i]$ starting at $v_2$ independent of $X$ (note that unlike $Y$, the random walk $Z$ does not leave $V_i$ and does not visit $\rho$).  By \cref{lem:mix:spec} and condition (3) of an $(\eps,\delta,n^{1.5})$-good decomposition (\cref{def:good:deco}) we have that for some $C=C(\delta)$ 
	\begin{equation*}
	\tmix^{G[V_i]}(\eps/2)+\floor{\sqrt{n}} \leq C\log(1/\eps)\frac{\sqrt{n}}{\theta}.
	\end{equation*}

Fix a constant
	$$A = \ceil{320\log(1/\eps)/\eps\delta^4} ,$$ 
so that by the last estimate, and since $v_2 \in V_i'$ the assertion of \cref{cl:stay:comp:good:deco} implies that 
	\begin{align}\label{eq:Ytakestimes}
	\pr_{v_2}\left(Y\left[0,2A(\tmix^{G[V_i]}(\eps/2)+\floor{\sqrt{n}})\right] \subseteq V_i\right) &\geq \pr_{v_2}\left(Y\left[0,\frac{2AC\log(1/\eps)\sqrt{n}}{\theta}\right] \subseteq V_i\right)\nonumber  \\&\geq 1-\frac{160AC\theta\eps^8}{\delta^3} - \pr\left(\tau_{\rho} \leq \frac{2AC\log(1/\eps)\sqrt{n}}{\theta}\right) \\\nonumber&\geq 1 - C\eps^2,
	\end{align}
for some $C = C(\delta)$. Thus we learn that with probability larger than $1-C\eps^2$ we can couple $Y$ and $Z$ such that $Y_t = Z_t$ for all $t\leq 2A(\tmix^{G[V_i]}(\eps/2)+\floor{\sqrt{n}})$. Now, since $v_1\in V_i'$ the assertion of \cref{cl:first:lerw} states that with probability at least $1-C\eps^2$ we have that $|\LE(X)|\geq \eps \sqrt{n}$ and $\LE(X)\subseteq V_i \cup \{\rho\}$. The minimal degree in $G[V_i]$ is at least $\delta^4n /40$ and $\delta n /2 \leq |V_i|\leq n$, allowing us to apply \cref{cl:hit:large:set} $A$ times together with the previous estimate to obtain that
	\begin{equation*}
	\pr_{v_2}(Z[0,2A(\tmix^{G[V_i]}(\eps/2)+\floor{\sqrt{n}})] \cap \LE(X) = \es) \leq \left(1-\frac{\eps\delta^4}{160}\right)^A + C\eps^2 \leq  (C+1) \eps^2 ,
	\end{equation*}
by our choice of $A$. This together with the coupling of $Y$ and $Z$ and \eqref{eq:Ytakestimes} gives
	\begin{equation*}
	\pr\left(Y_{\LE(X)} \neq \rho \pand \LE(Y) \subset V_i \right) \geq 1-C\eps^2.
	\end{equation*}

We are left with bounding $|\LE(Y)|$ from below. We will show that with large probability $Y[0,\eps^8\theta\sqrt{n})\subseteq \LE(Y)$. Indeed, since $\E_{v_1} \tau_\rho=\sqrt{n}/\theta\eps^{4}$ we have that $|\LE(X)|\leq \sqrt{n}/\theta\eps^6$ with probability at least $1-\eps^2$. Furthermore, since $v_1\neq v_2$ and the degree is at least $\delta n$ we have that $v_2\not \in \LE(X)$ with probability at least $1-\eps^2-\delta^{-1}n^{-1/2}/\theta \eps^6 \geq 1-C\eps^2$. Hence
	\begin{align}\label{eq:one:bound:avoid:U}
		\pr(Y[0,\eps^8\theta\sqrt{n}) \cap \LE(X) \neq \es) &\leq \eps^8\theta\sqrt{n} \frac{\sqrt{n}/\theta \eps^6}{\delta n} + C\eps^2  \leq C\eps^2.
	\end{align} 
Furthermore, since the degree is at least $\delta n$, the union bound gives that
		\begin{equation}\label{eq:prefix:hit:suffix}
		\pr(Y[0,\eps^{8}\theta\sqrt{n}) \cap Y[\eps^{8}\theta\sqrt{n},
		\sqrt{n}/\theta\eps^{6}) \neq\es )
		\leq \frac{\eps^2}{\delta}.
		\end{equation} 
Since $\tau_\rho \geq \sqrt{n}/\theta \eps^6$ occurs with probability at most $\eps^2$, we deduce by \eqref{eq:one:bound:avoid:U}, \eqref{eq:prefix:hit:suffix} that 
		\begin{equation*}
			\pr\left( \LE(Y[0,\eps^{8}\theta\sqrt{n})) \subseteq \LE(Y)\right) \geq 1 - C\eps^2.
		\end{equation*}
for some $C=C(\delta)<\infty$. Lastly, again by the linear minimal degree and the union bound, the probability that there is a repeating vertex in $Y[0,\eps^{8}\theta\sqrt{n})$ is at most $\eps^{16}\theta^2/\delta$; when this does not occur $Y[0,\eps^{8}\theta\sqrt{n})=\LE(Y[0,\eps^{8}\theta\sqrt{n}))$, concluding our proof. \end{proof}

\begin{proof}[Proof of \cref{lem:paths:in:deco}]
	If $\theta\eps^8 \leq 1 / \sqrt{n}$, the claim is trivial. We assume the converse is true, and we take the vertices $v_1^i$ and $v_2^i$ from \cref{lem:path:builder}. We denote by $\varphi$ and $\varphi'$ the paths between them in $\UST(G)$ and $\UST(G^\rho)$, respectively. As mentioned in the beginning of this subsection, we couple $\UST(G)$ and $\UST(G^\rho)$ such that $\UST(G^\rho)\cap E(G)\subseteq\UST(G)$. We use \eqref{eq:dtv:original:path} and recall that under this coupling, if $\rho\notin \varphi'$, then $\varphi=\varphi'$. Hence it suffices to show that
	\begin{equation*}
			\pr\left(\eps^8\theta\sqrt{n} \leq |\varphi'| \leq \frac{\sqrt{n}}{\theta\eps^{8}} \pand \varphi' \subseteq V_i \pand \rho\notin\varphi' \right) \geq 1-C\eps^2,
	\end{equation*}
	for some $C = C(\delta)$.
	By Wilson's algorithm, we can sample $\varphi'$ by sampling $\LE(X)$, a $\LERW$ from $v_1^i$ to $\rho$ and then sampling $\LE(Y)$, another $\LERW$ from $v_2^i$ to $\LE(X)$. The path between $v_1^i$ and $v_2^i$ in $\LE(X)\cup\LE(Y)$ is distributed as the path between $v_1^i$ and $v_2^i$ in $\UST(G^\rho)$.
	By \cref{lem:path:builder} and \cref{cl:first:lerw}, this path is contained in $V_i$ and does not contain $\rho$ with probability larger than $1-C\eps^2$. By construction $|\varphi'|\geq \LE(Y)$ hence \cref{lem:path:builder}  gives the required lower bound on $|\varphi'|$. Finally, as the length of $\LE(X)$ and $\LE(Y)$ is bounded by two independent random variables with the distribution of $\tau_{\rho}$, by Markov's inequality
	\begin{equation*}
	\pr\left(|\LE(X)| + |\LE(Y)| \geq \frac{\sqrt{n}}{\theta\eps^8}\right) \leq 1-C\eps^4,
	\end{equation*}
concluding the proof. 
\end{proof}

\section{Proof of main theorem} \label{sec:proofmainthm}

In \cite{MNS19+}, the following strategy was used to show that the diameter grows like $\sqrt{n}$. First, a small part of the $\UST$ is sampled. This part contains roughly $\sqrt{n}$ vertices (in \cite{MNS19+}, it is simply a path between two vertices). Then, it is shown that this part of the $\UST$ is difficult to avoid in the sense that random walks starting from any vertex of the graph will hit it with positive probability within roughly $\sqrt{n}$ steps. To formalize and quantify this we first define
\begin{equation*}
\vect{p}^t_W(v,v) = \pr_v(X_t = v, \tau_W > t) ,
\end{equation*} 
for any $W\subset V$. Next we define the \defn{$W$-bubble sum} by
\begin{equation*}
\bub_W(G) = \sum_{t=0}^{\infty}(t+1)\sup_{v\in G}\vect{p}^t_W(v,v).
\end{equation*}
If the set $W$ is difficult to avoid, then the $\vect{p}^t_W(v,v)$ decays fast with $t$ and thus $\bub_W(G)$ is small. It is shown in \cite{MNS19+} that if $\bub_W(G)$ is small, then the diameter of $\UST(G/W)$ cannot be too large:

\begin{lemma}[\cite{MNS19+}*{Lemma 3.13}]\label{lem:mns:lemma}
	Let $G = (V,E)$ be a connected graph, let $D=\frac{\max_v \deg(v)}{\min_v \deg(v)}$ and let $W$ be a non-empty vertex set. 
	Let $\cT_W$ be a $\UST$ on the graph $G/W$. Then
	\begin{equation*}
	    \pr(\diam(\cT_W) \geq \ell) \leq \frac{C_3|W|}{\ell},
	\end{equation*}
	for $C_3 = 138420\cdot D^4\bub_W(G)^3\log(192D\bub_W(G))$.
\end{lemma}

In our context, we will take $W$ to be the union of $k$ paths in the $\UST$ drawn according to an $(\eps,\delta,n^{1.5})$-good decomposition. In the next few claims, using the results we obtained in \cref{sec:rw}, we will show that with high probability $\bub_W(G)=O(1)$, after which we will prove \cref{thm:main}.

\begin{lemma}\label{lem:pbub} For any $\delta>0$ there exists $b(\delta)>0$ such that for any $\eps \in (0,b)$ there exists $C = C(\eps, \delta)<\infty$ and $c=c(\eps,\delta)>0$ such that the following holds. Let $G=(V,E)$ be a connected simple graph on $n$ vertices with minimal degree at least $\delta n$ and $\eps>0$. Denote by $V=V_1 \sqcup \ldots \sqcup V_k$ an $(\eps,\delta,n^{1.5})$-good decomposition of $G$ with parameter $\theta$ (as guaranteed to exist by \cref{lem:main:deco}). Then, for every set $W$ that satisfies $|W\cap V_i| \geq \eps^{8}\theta\sqrt{n}$ for every $i\in[k]$, we have that
	\begin{equation*}
	\sum_{t=1}^{\infty}(t+1)\sup_{v\in G} \vect{p}^t_W(v,v) \leq C.
	\end{equation*}
\end{lemma}
\begin{proof} Let $W$ be such a set and fix $v\in V$. We will first show that 
\begin{equation}\label{eq:pbubtoshow}
		\pr_v\left(X\left[0,C\sqrt{n}\right] \cap W \neq \es\right) \geq c.
\end{equation}
for some $C,c$ depending on $\eps$ and $\delta$. 
There exists at least one component $V_i$ in the decomposition such that $\pr_v(X_1 \in V_i) \geq 1/k$ (note that $v$ does not necessarily belong to $V_i$). Let $Y$ be a random walk on $G[V_i]$ starting from some $u\in V_i$. By condition (4) of an $(\eps,\delta,n^{1.5})$-good decomposition (\cref{def:good:deco}), the minimal degree of $G[V_i]$ is at least $\delta^4n/40$. We apply \cref{cl:hit:large:set} to the graph $G[V_i]$, with $\eps^{8}\theta$ playing the role of $\eps$ in the claim, to obtain that for every $u\in V_i$
	\begin{equation}\label{eq:prob:to:hit:W}
	\pr_u \Big (Y[0,2(\tmix^{G[V_i]}(\eps^{8}\theta) + \floor{\sqrt{n}})] \cap W \neq \es\Big ) \geq \frac{\eps^{8}\theta\delta^{4}}{160}.
	\end{equation}
	By definition of an $(\eps,\delta,n^{1.5})$-good decomposition (\cref{def:good:deco}) we have $\theta \geq \eps^{11\cdot 2^{2/\delta}}$. Hence by \eqref{eq:exp:mix}, \cref{lem:mix:spec} and condition (3) of \cref{def:good:deco} we get
	\begin{equation}\label{eq:up:bound:mix:time}
	\tmix^{G[V_i]}(\eps^{8}\theta) + \floor{\sqrt{n}} \leq (8+11 \cdot 2^{2/\delta})(\tmix^{G[V_i]}(\eps/2)) + \floor{\sqrt{n}} \leq \frac{B\log(1/\eps)\sqrt{n}}{\theta},
	\end{equation}  
	for some $B = B(\delta)$. By \cref{cl:stay:comp:good:deco}
	\begin{equation}\label{eq:couple:failure:hit:W}
	\pr\Big (X_1 \in V_i \pand \exists t \leq 2(\tmix^{G[V_i]}(\eps^{8}\theta) + \floor{\sqrt{n}}) \hbox{ with }  X_t\notin V_i \Big ) \leq \frac{2B\theta\log(1/\eps)\eps^9}{\delta^2}.
	\end{equation}
	Hence, conditioned on $X_1 \in V_i$, if we set $Y_0 = X_1$ then we can couple these two walks such that $Y_t = X_{t+1}$ for all $t\leq2(\tmix^{G[V_i]}(\eps^{8}\theta) + \floor{\sqrt{n}})$ with failure probability bounded by the right-hand side of \eqref{eq:couple:failure:hit:W}. This and \eqref{eq:prob:to:hit:W} imply that
	\begin{align*}
	&\pr_v\Big (X[0,2(\tmix^{G[V_i]}(\eps^{8}\theta) + \floor{\sqrt{n}})	] \cap W \neq \es\Big ) \geq {1\over k} \Big (\frac{\eps^{8}\theta\delta^{4}}{160} - \frac{2B\theta\log(1/\eps)\eps^9}{\delta^2} \Big ) .
	\end{align*}
	Plugging in \eqref{eq:prob:to:hit:W} and \eqref{eq:couple:failure:hit:W} we obtain that the right-hand side is bounded from below by
	\begin{align*}
	\frac{1}{k}\left(\frac{\eps^{8}\theta\delta^{4}}{160} - \frac{2B\theta\log(1/\eps)\eps^9}{\delta^2}\right),
	\end{align*}
	which is lower bounded by some $c = c(\eps,\delta)$. Now \eqref{eq:pbubtoshow} follows by \eqref{eq:up:bound:mix:time} and taking $C=\frac{B\log(1/\eps)}{\theta}$. Now by \eqref{eq:pbubtoshow} and the Markov property, for any positive integer $m$ and any $t\in [mC\sqrt{n},(m+1)C\sqrt{n}]$ we have
	\begin{equation*}
	\vect{p}_W^t(v,v) \leq \frac{(1-c)^m}{\delta n} ,
	\end{equation*} 
	where the denominator accounts for the last step returning to $v$. We conclude that 
	\begin{align*}
	\bub_W(G) = \sum_{t=0}^{\infty}(t+1)\vect{p}^t_W(v,v) 
	&\leq \sum_{m=0}^{\infty}\sum_{t=mC\sqrt{n}}^{(m+1)C\sqrt{n}}((m+1)C\sqrt{n})\frac{(1-c)^m}{\delta n} 
	\\&\leq \frac{C^2}{\delta}\sum_{m=0}^{\infty}(m+1)\left(1 - c\right)^m ,
	\end{align*}
and this concludes our proof since the infinite sum above converges.
\end{proof}

\begin{proof}[Proof of \cref{thm:main}]
	Let $\eps>0$ and let $G=(V,E)$ be a connected simple graph on $n$ vertices with minimal degree at least $\delta n$. By \cref{lem:main:deco}, there exists an $(\eps,\delta,n^{1.5})$ good decomposition of $G$, denoted by $V=V_1\sqcup \ldots \sqcup V_k$. By \cref{lem:paths:in:deco} there exist some $C' = C'(\delta)$ and $k$ pairs of distinct vertices $v_1^i,v_2^i \in V_i$, such that if $\varphi_i$ is the random path between $v_1^i$ and $v_2^i$ in $\UST(G)$, then
	\begin{equation}\label{eq:k:paths:prob}
		\pr\left(\forall i\in[k] : \eps^8\theta\sqrt{n} \leq |\varphi_i| \leq \frac{\sqrt{n}}{\theta\eps^{8}} \pand \varphi_i \subseteq V_i	 \right) \geq 1-C'k\eps^2.
	\end{equation} 
	We condition on this event and on the collection of paths $(\varphi_i)_{i\in[k]}$ and denote by $W$ the set of vertices of $(\varphi_i)_{i\in[k]}$. Let $H$ be the graph obtained from $G$ by contracting each $\varphi_i$ into a single vertex. Then \cref{lem:contract} implies that $\UST(H) \cup \{\varphi_i\}_{i\in[k]}$ has the distribution of $\UST(G)$. Hence $\diam(\UST(G)) \leq \diam(\UST(H)) + |W|$.
	Denote also by $\cT_W$ the $\UST$ on $G/W$. By \cref{lem:neg:ass:contract}, we have that $\UST(H)$ stochastically dominates $\cT_W := \UST(G/W)$ (when viewed as random subsets of $E(G)$) hence there is a coupling such that $\cT_W \subseteq \UST(H)$ and since $H$ has $k-1$ vertices more than $G/W$ we deduce that $\UST(H)$ is a union of $\cT_W$ and at most $k-1$ more edges. Hence the diameter of $\UST(H))$ is at most $k-1$ times the diameter of $\cT_W$. We conclude that
	$$ \diam(\UST(G)) \leq (k-1)\diam(\cT_W) + |W| \, ,$$

	Now, \cref{lem:pbub} implies that the set $W$ has $\bub_W(G) \leq C$ for some $C=C(\eps,\delta)$ so that \cref{lem:mns:lemma} gives
	\begin{equation}\label{eq:height:bound}
	    \pr(\diam(\cT_W) \geq \ell) \leq \frac{C_3|W|}{\ell},
	\end{equation}
	for $C_3 = C_3(\eps,\delta)$ and any $\ell \geq 1$. Hence under our conditioning we have that for any $A>k/\theta \eps^8$
	$$ \pr(\UST(G)\geq A\sqrt{n}) \leq \pr \Big ( \diam(\cT_W) \geq (A-k/\theta \eps^8)\sqrt{n}/(k-1) \Big ) \leq \frac{C_3 k(k-1)}{ \theta \eps^8(A-k/\theta \eps^8)} \, ,$$
	which together with \eqref{eq:k:paths:prob} concludes the proof. \end{proof}

\section*{Acknowledgments}
NA is supported in part by
NSF grant DMS-1855464, BSF grant 2018267 and the Simons Foundation. AN and MS are supported by ISF grants 1207/15 and 1294/19 as well as ERC starting grant 676970 RANDGEOM. We thank Asaf Shapira for useful discussions and his assistance in proving \cref{cl:prim:deco}, and also Majid Farhadi, Suprovat Ghoshal, Anand Louis, and Prasad Tetali for allowing us to present their alternate proof of \cref{thm:cheeger:is:everything}, see \cref{rmrk:alternate}.

\bibliography{library}
\end{document}